\documentclass[12pt,a4paper]{amsart}
\usepackage{amssymb,amsxtra,eucal}
\usepackage[all,cmtip]{xy}

\calclayout

\newcommand{\+}{\protect\nobreakdash-}
\renewcommand{\:}{\colon}

\newcommand{\rarrow}{\longrightarrow}

\newcommand{\ot}{\otimes}
\newcommand{\ocn}{\odot}
\newcommand{\tim}{\rightthreetimes}

\newcommand{\lrarrow}{\mskip.5\thinmuskip\relbar\joinrel\relbar\joinrel
 \rightarrow\mskip.5\thinmuskip\relax}

\newcommand{\bu}{{\text{\smaller\smaller$\scriptstyle\bullet$}}}

\newcommand{\ilim}
 {\mathop{\text{\normalfont``$\varinjlim$''\!\!}}\nolimits}

\DeclareMathOperator{\Hom}{Hom}
\DeclareMathOperator{\Ext}{Ext}
\DeclareMathOperator{\Tor}{Tor}
\DeclareMathOperator{\Spec}{Spec}
\DeclareMathOperator{\Tot}{Tot}

\newcommand{\Modl}{{\operatorname{\mathsf{--Mod}}}}
\newcommand{\Modr}{{\operatorname{\mathsf{Mod--}}}}
\newcommand{\Discr}{{\operatorname{\mathsf{Discr--}}}}
\newcommand{\Bimod}{{\operatorname{\mathsf{--Bimod--}}}}
\newcommand{\Comodl}{{\operatorname{\mathsf{--Comod}}}}
\newcommand{\Contra}{{\operatorname{\mathsf{--Contra}}}}

\newcommand{\qcoh}{{\operatorname{\mathsf{--qcoh}}}}

\newcommand{\dfl}{{\operatorname{\mathsf{-flat}}}}

\newcommand{\Sets}{\mathsf{Sets}}
\newcommand{\Fun}{\mathsf{Fun}}
\newcommand{\Com}{\mathsf{Com}}
\newcommand{\Ac}{\mathsf{Ac}}

\newcommand{\Ab}{\mathsf{Ab}}

\renewcommand{\flat}{\mathsf{flat}}
\renewcommand{\cot}{\mathsf{cot}}
\newcommand{\sep}{\mathsf{sep}}

\newcommand{\id}{{\mathrm{id}}}

\newcommand{\B}{\mathcal B}
\newcommand{\C}{\mathcal C}
\newcommand{\D}{\mathcal D}
\newcommand{\F}{\mathcal F}
\newcommand{\G}{\mathcal G}
\newcommand{\HH}{\mathcal H}

\newcommand{\LL}{\mathcal L}
\newcommand{\M}{\mathcal M}
\newcommand{\N}{\mathcal N}
\newcommand{\cO}{\mathcal O}

\newcommand{\R}{\mathfrak R}
\newcommand{\fP}{\mathfrak P}
\newcommand{\fQ}{\mathfrak Q}
\newcommand{\fF}{\mathfrak F}
\newcommand{\fG}{\mathfrak G}
\newcommand{\fH}{\mathfrak H}
\newcommand{\fI}{\mathfrak I}
\newcommand{\fJ}{\mathfrak J}
\newcommand{\fM}{\mathfrak M}
\newcommand{\fN}{\mathfrak N}
\newcommand{\fA}{\mathfrak A}
\newcommand{\fB}{\mathfrak B}
\newcommand{\fU}{\mathfrak U}
\newcommand{\fX}{\mathfrak X}
\newcommand{\fZ}{\mathfrak Z}

\newcommand{\sB}{\mathsf B} 
\newcommand{\sC}{\mathsf C}
\newcommand{\sD}{\mathsf D}
\newcommand{\sE}{\mathsf E}
\newcommand{\sF}{\mathsf F}
\newcommand{\sG}{\mathsf G}
\newcommand{\sK}{\mathsf K}
\newcommand{\sL}{\mathsf L}
\newcommand{\sM}{\mathsf M}
\newcommand{\sS}{\mathsf S}

\newcommand{\boZ}{\mathbb Z}

\newcommand{\Section}[1]{\bigskip\section{#1}\medskip}
\theoremstyle{plain}
\newtheorem{thm}{Theorem}[section]

\newtheorem{lem}[thm]{Lemma}
\newtheorem{prop}[thm]{Proposition}
\newtheorem{cor}[thm]{Corollary}
\theoremstyle{definition}
\newtheorem{ex}[thm]{Example}

\newtheorem{rem}[thm]{Remark}

\newtheorem{qst}[thm]{Question}
\newtheorem{qsts}[thm]{Questions}

\begin{document}

\title{Flat comodules and contramodules as directed colimits,
and cotorsion periodicity}

\author{Leonid Positselski}

\address{Institute of Mathematics, Czech Academy of Sciences \\
\v Zitn\'a~25, 115~67 Praha~1 \\ Czech Republic}

\email{positselski@math.cas.cz}

\begin{abstract}
 This paper is a follow-up to~\cite{PS6}.
 We consider two algebraic settings of comodules over a coring and
contramodules over a topological ring with a countable base of
two-sided ideals.
 These correspond to two (noncommutative) algebraic geometry settings
of certain kind of stacks and ind-affine ind-schemes.
 In the context of a coring $\C$ over a noncommutative ring $A$, we show
that all $A$\+flat $\C$\+comodules are $\aleph_1$\+directed colimits
of $A$\+countably presentable $A$\+flat $\C$\+comodules.
 In the context of a complete, separated topological ring $\R$
with a countable base of neighborhoods of zero consisting of
two-sided ideals, we prove that all flat $\R$\+contramodules are
$\aleph_1$\+directed colimits of countably presentable flat
$\R$\+contramodules.
 We also describe arbitrary complexes, short exact sequences, and pure
acyclic complexes of $A$\+flat $\C$\+comodules and flat
$\R$\+contramodules as $\aleph_1$\+directed colimits of similar
complexes of countably presentable objects.
 The arguments are based on a very general category-theoretic technique
going back to an unpublished 1977 preprint of Ulmer and rediscovered
in~\cite{Pacc}.
 Applications to cotorsion periodicity and coderived categories of
flat objects in the respective settings are discussed.
 In particular, in any acyclic complex of cotorsion $\R$\+contramodules,
all the contramodules of cocycles are cotorsion.
\end{abstract}

\maketitle

\tableofcontents

\section*{Introduction}
\medskip

 The classical \emph{Govorov--Lazard theorem}~\cite{Gov,Laz} says
that all flat modules (over an arbitrary associative ring~$R$) are
directed colimits of projective modules, and in fact, even of finitely
generated free $R$\+modules.
 In the context of algebraic geometry, over a nonaffine scheme $X$,
there are usually \emph{not} enough projective objects in the abelian
category of quasi-coherent sheaves $X\qcoh$, which makes the role of
flat quasi-coherent sheaves even more important than the role of
flat modules in module theory.
 What should a suitable version of the Govorov--Lazard theorem say
about flat quasi-coherent sheaves?

 One could try to use locally free or locally projective quasi-coherent
sheaves in the role of projective modules.
 But it is still an open problem whether there are enough locally
projective quasi-coherent sheaves on $X$ under any reasonable
assumptions on a scheme~$X$ \,\cite{To}.
 On the other hand, it is known that there are enough flat
quasi-coherent sheaves on any quasi-compact semi-separated
scheme~\cite[Section~2.4]{M-n}, \cite[Lemma~A.1]{EP}.
 So describing flat quasi-coherent sheaves is a worthwhile undertaking.

 The approach in this paper follows the idea that, in many homological
algebra contexts, one can use objects of finite projective dimension
in lieu of projective ones.
 It is known that any countably presentable flat $R$\+module has
projective dimension~$\le1$ \,\cite[Corollary~2.23]{GT}.
 Thus our suggested answer to the question in the first paragraph is
this: the Govorov--Lazard theorem in algebraic geometry should tell us
that any flat sheaf is a directed colimit of locally countably
presentable flat ones.
 Notice, for comparison, that any finitely presentable flat module is
projective.

 A partial result was obtained in the paper~\cite{EGO}, where it was
shown that, under certain additional assumptions on a quasi-compact
semi-separated scheme $X$, any flat quasi-coherent sheaf on $X$ is
a directed colimit of locally countably generated flat quasi-coherent
sheaves locally of projective dimension~$\le1$
\,\cite[Theorem~B or Theorem~4.9]{EGO}.
 In full generality, the desired assertion for quasi-compact
quasi-separated schemes (or even more generally, for countably
quasi-compact, countably quasi-separated schemes) $X$ was proved in
the preprint~\cite{PS6}: any flat quasi-coherent sheaf on $X$ is
an $\aleph_1$\+directed colimit of locally countably presentable
flat quasi-coherent sheaves~\cite[Theorem~2.4 or Theorem~3.5]{PS6}.
 Here an \emph{$\aleph_1$\+directed colimit} means the directed colimit
of a diagram indexed by an $\aleph_1$\+directed poset, i.~e., a poset
in which every countable subset has an upper bound.

 The present paper aims to extend the results of~\cite{PS6} to two
algebraic geometry settings more general than schemes: the stacks
and the ind-schemes.
 In fact, we don't assume our rings to be commutative in this paper,
so it also extends the results of~\cite{PS6} into certain noncommutative
algebraic geometry realms.

 The observation that certain noncommutative stacks $X$ can be described
by corings $\C$ over noncommutative rings $A$ is due to Kontsevich
and Rosenberg~\cite[Section~2]{KR}, \cite{KR2}.
 The quasi-coherent sheaves on $X$ are then interpreted as left
$\C$\+comodules (which form an abelian category whenever $\C$ is a flat
right $A$\+module).
 The quasi-coherent sheaf corresponding to a $\C$\+comodule $\M$ is
flat if and only if $\M$ is a flat $A$\+module.

 In commutative algebraic geometry, this description applies to
the stacks $X$ admitting a flat affine surjective morphism $U\rarrow X$
from an affine scheme~$U$.
 Specifically, one has $A=\cO(U)$ and $\C=\cO(U\times_XU)$.
 For example, if $X$ is a quasi-compact semi-separated scheme covered
by a finite collection of affine open subschemes $U_\alpha\subset X$,
then one can take $U$ to be the disjoint union
$U=\coprod_\alpha U_\alpha$.
 Then $A=\bigoplus_\alpha\cO(U_\alpha)$ and
$\C=\bigoplus_{\alpha,\beta}\cO(U_\alpha\cap U_\beta)$.
 The category of quasi-coherent sheaves on $X$ is equivalent to
the category of $\C$\+comodules, and the inverse image functor
$X\qcoh\rarrow U\qcoh$ corresponds to the forgetful functor
$\C\Comodl\rarrow A\Modl$.
 The $\C$\+comodule structure on a given $A$\+module describes
the descent/gluing datum needed to glue a quasi-coherent sheaf on $X$
from a given collection of quasi-coherent sheaves on~$U_\alpha$.

 Let $\C$ be a coring over a noncommutative ring~$A$ (in the sense of
the paper~\cite{Swe} and the book~\cite{BW}).
 Then we prove that any $\C$\+comodule that is flat as an $A$\+module
can be obtained as an $\aleph_1$\+directed colimit of $\C$\+comodules
that are flat and countably presentable as $A$\+modules
(see Theorem~\ref{flat-comodules-as-directed-colimits-theorem}).
 Similarly, any $\C$\+comodule is an $\aleph_1$\+directed colimit of
$\C$\+comodules that are countably presentable as $A$\+modules
(Remark~\ref{arbitrary-comodules-aleph1-loc-presentable}).
 Furthermore, any complex of $A$\+flat $\C$\+comodules is
an $\aleph_1$\+directed colimit of complexes of $A$\+countably
presentable $A$\+flat $\C$\+comodules
(by Proposition~\ref{complexes-of-flat-comodules-prop}).
 We also show that any short exact sequence of $A$\+flat $\C$\+comodules
is an $\aleph_1$\+directed colimit of short exact sequences of
$A$\+countably presentable $A$\+flat $\C$\+comodules (this is the result
of our Proposition~\ref{short-exact-sequences-of-flat-comodules}).
 More generally, any $A$\+pure acyclic complex of $A$\+flat
$\C$\+comodules is an $\aleph_1$\+directed colimit of $A$\+pure acyclic
complexes of $A$\+countably presentable $A$\+flat $\C$\+comodules (see
Corollary~\ref{comods-pure-acycl-complexes-as-aleph1-dir-colims}).

 A left $R$\+module $P$ is said to be \emph{cotorsion} if
$\Ext^1_R(F,P)=0$ for all flat left $R$\+modules~$F$.
 The \emph{cotorsion periodicity theorem} of Bazzoni,
Cort\'es-Izurdiaga, and Estrada~\cite{BCE} claims that, in any
acyclic complex of cotorsion $R$\+modules, all the modules of
cocycles are also cotorsion~\cite[Theorem~1.2(2),
Proposition~4.8(2), or Theorem~5.1(2)]{BCE}.

 The following comodule version of cotorsion periodicity theorem is
obtained in this paper.
 Assume that $\C$ is a flat right $A$\+module, all left $\C$\+comodules
are quotients of $A$\+flat left $\C$\+comodules, and all left
$\C$\+comodules having finite projective dimension as left $A$\+modules
also have finite projective dimension as $\C$\+comodules.
 Let us say that a left $\C$\+comodule $\B$ is \emph{cotorsion} if
$\Ext^1_\C(\F,\B)=0$ for all $A$\+flat left $\C$\+comodules~$\F$.
 Here $\Ext^*_\C$ denotes the $\Ext$ groups in the abelian category of
left $\C$\+comodules $\C\Comodl$.
 Then, in any acyclic complex of cotorsion left $\C$\+comodules,
the comodules of cocycles are also cotorsion
(see Theorem~\ref{comodule-cotorsion-periodicity-theorem}
and Corollary~\ref{comodule-cotorsion-cocycles-cor}).
 As a corollary of this periodicity theorem, we conclude that
the derived category of the abelian category of left $\C$\+comodules
is equivalent to the derived category of the exact category of
cotorsion left $\C$\+comodules
(Corollary~\ref{comodule-cotorsion-derived-equivalence}).

 Under the same assumptions as in the previous paragraph, we also
obtain the following description of $A$\+pure acyclic complexes of
$A$\+flat left $\C$\+comodules.
 Such complexes are $\aleph_1$\+directed colimits of totalizations of
finite acyclic complexes (of a certain fixed length) of complexes of
$A$\+countably presentable $A$\+flat $\C$\+comodules.
 This result, based on~\cite[Lemma~2.1]{Psemi} and~\cite[proof of
Proposition~8.8]{Pedg}, provides a comodule version of the well-known
description of pure acyclic complexes of flat $R$\+modules as
directed colimits of finite contractible complexes of finitely
generated projective
$R$\+modules~\cite[Theorem~2.4\,(1)\,$\Leftrightarrow$\,(3)]{EG},
\cite[Theorem~8.6\,(ii)\,$\Leftrightarrow$\,(iii)]{Neem}.
 This is our
Theorem~\ref{flat-comodule-pure-acyclic-complexes-characterized}.

 The discussion of ind-schemes in this paper is mostly restricted to
strict ind-affine $\aleph_0$\+ind-schemes, i.~e., the ind-schemes
represented by countable directed diagrams of closed immersions of
affine schemes.
 The category of such ind-schemes $\mathfrak X$ is anti-equivalent to
the category of complete, separated topological commutative rings $\R$
with a countable base of neighborhoods of zero consisting of open
ideals~\cite[Example~7.11.2(i)]{BD2}, \cite[Example~1.6(2)]{Psemten}.
 Flat pro-quasi-coherent pro-sheaves on $\mathfrak X$ (in the sense
of~\cite[Section~7.11.3]{BD2}, \cite[Section~3.4]{Psemten}) are
described by flat $\R$\+contramodules, while the category of all
pro-quasi-coherent pro-sheaves on $\mathfrak X$ is equivalent to
the category of \emph{separated}
$\R$\+contramodules~\cite[Examples~3.8(1\+-2)]{Psemten}.

 As a noncommutative generalization of this class of ind-schemes, we
consider complete, separated topological associative rings $\R$ with
a countable base of neighborhoods of zero consisting of open
\emph{two-sided} ideals.
 Contramodules over such topological rings $\R$ were discussed
in the long preprint~\cite[Appendix~E]{Pcosh}, while the more general
case of a countable base of \emph{right} ideals was studied in
the paper~\cite{PR}.

 In this context (assuming a countable topology base of two-sided
ideals in~$\R$), we prove that any flat $\R$\+contramodule is
an $\aleph_1$\+directed colimit of countably presentable flat
$\R$\+contramodules (see
Theorem~\ref{flat-contramodules-as-directed-colimits-theorem}).
 Furthermore, any complex of flat $\R$\+contramodules is
an $\aleph_1$\+directed colimit of complexes of countably presentable
flat $\R$\+contramodules (by
Proposition~\ref{complexes-of-flat-contramodules-prop}).
 We also show that any short exact sequence of flat
$\R$\+contramodules is an $\aleph_1$\+directed colimit of short exact
sequences of countably presentable flat $\R$\+contramodules
(Proposition~\ref{short-exact-sequences-of-flat-contramodules}), and
any acyclic complex of flat $\R$\+contramodules with flat
$\R$\+contramodules of cocycles is an $\aleph_1$\+directed colimit of
such complexes of countably presentable flat $\R$\+contramodules
(Corollary~\ref{contramods-pure-acycl-cplxs-as-aleph1-dir-colims}).
 It follows that any acyclic complexes of flat $\R$\+contramodules
with flat $\R$\+contramodules of cocycles is an $\aleph_1$\+directed
colimit of totalizations of short exact sequences of complexes of
flat $\R$\+contramodules
(Theorem~\ref{flat-contramodule-pure-acyclic-complexes-charact-ed}).

 Then we deduce the following contramodule version of cotorsion
periodicity theorem.
 A left $\R$\+contramodule $\fB$ is said to be \emph{cotorsion} if
$\Ext^{\R,1}(\fF,\fB)=0$ for all flat left $\R$\+contramodules~$\fF$.
 Here $\Ext^{\R,*}$ denotes the $\Ext$ groups in the abelian
category of left $\R$\+contramodules $\R\Contra$.
 The theorem claims that, in any acyclic complex of cotorsion
$\R$\+contramodules, the contramodules of cocycles are also cotorsion
(see Theorem~\ref{contramodule-cotorsion-periodicity-theorem}
and Corollary~\ref{contramodule-cotorsion-cocycles-cor}).
 As a corollary of this periodicity theorem, we deduce an equivalence
between the derived category of the abelian category of left
$\R$\+contramodules and the derived category of the exact category of
cotorsion left $\R$\+contramodules
(Corollary~\ref{contramodule-cotorsion-derived-equivalence}).

 Periodicity theorems can be thought of as expressing special
properties of directed colimit closures of exact categories (see
the paper~\cite{Plce} for a discussion of the general concept of
such directed colimit closure).
 We refer to the introduction to~\cite{BHP} for a quick survey on
periodicity theorems and to the preprint~\cite[Sections~7.8
and~7.10]{Pphil} for a discussion of the mentioned point of view
on periodicity theorems.
 In particular, the category of flat modules is the directed colimit
closure of the category of finitely generated projective modules,
which is split exact.
 In this context, the flat/projective periodicity theorem~\cite{BG,Neem}
can be interpreted as saying that \emph{for the exact category of flat
modules, the contraderived category coincides with the derived
category}~\cite[Theorem~7.14]{Pphil}, while the cotorsion periodicity
theorem~\cite{BCE} is closely related to the assertion that \emph{for
the exact category of flat modules, the coderived category coincides
with the derived category}~\cite[Theorem~7.18]{Pphil}.

 We refer the reader to the survey paper~\cite[Section~7]{Pksurv} for
a discussion of the history and philosophy of the coderived and
contraderived categories (see also~\cite[Remark~9.2]{PS4}).
 More advanced discussions of the coderived and contraderived categories
in the context relevant to the present paper can be found in
the papers~\cite{PS4,PS5} (see~\cite{Pedg} for a different point
of view).
 In this paper we observe that, under suitable assumptions, the results
describing the pure acyclic complexes of flat objects as the directed
colimits of pure acyclic complexes of countably presentable flat objects
can be interpreted as implying that \emph{for the exact category of
such flat objects, the coderived category coincides with the derived
category}.
 The description of the coderived category obtained
in~\cite[Corollary~0.5 or Proposition~8.13]{PS5} is used as
the reference point for the analogy or comparison here, in connection
with our
Theorems~\ref{flat-comodule-pure-acyclic-complexes-characterized}
and~\ref{flat-contramodule-pure-acyclic-complexes-charact-ed}.

 Let us say a few words about the proofs.
 The proofs of the main results of this paper are based on very
general category-theoretic observations concerning preservation of
$\kappa$\+accessible categories and $\kappa$\+presentable objects
(for a regular cardinal~$\kappa$) by category-theoretic constructions
such as the pseudopullback~\cite[Proposition~3.1]{CR},
\cite[Pseudopullback Theorem~2.2]{RR},
the equifier~\cite[Section~3]{Pacc}, and
the inserter~\cite[Section~4]{Pacc}.
 These results, going back to an unpublished 1977 preprint of
Ulmer~\cite{Ulm} and rediscovered in~\cite{Pacc}, depend on
the assumption of existence of a smaller infinite cardinal
$\lambda<\kappa$ such that the $\kappa$\+accessible categories involved
have colimits of $\lambda$\+indexed chains and the functors involved
preserve such colimits.
 In the situation at hand, we take $\kappa=\aleph_1$ and
$\lambda=\aleph_0$.
 The key observation is that the classes of $A$\+flat $\C$\+comodules
and flat $\R$\+contramodules, as well as various classes of complexes
of flat comodules and flat contramodules, are preserved by all directed
colimits.
 This allows to prove that the respective categories are
$\aleph_1$\+accessible, and describe their full subcategories
of $\aleph_1$\+presentable objects.

\subsection*{Acknowledgement}
 I~am grateful to Jan \v St\!'ov\'\i\v cek for helpful discussions.
 The author is supported by the GA\v CR project 23-05148S and
the Czech Academy of Sciences (RVO~67985840).

\Section{Preliminaries on Accessible Categories}
\label{accessible-preliminaries-secn}

 Let $\kappa$~be a regular cardinal.
 We refer to~\cite[Definition~1.4, Theorem~1.5, Corollary~1.5,
Definition~1.13(1), and Remark~1.21]{AR} for the discussion of
\emph{$\kappa$\+directed posets} vs.\ \emph{$\kappa$\+filtered small
categories}, and accordingly, $\kappa$\+directed vs.\
$\kappa$\+filtered colimits.

 Let $\sC$ be a category with $\kappa$\+directed (equivalently,
$\kappa$\+filtered) colimits.
 An object $S\in\sC$ is said to be \emph{$\kappa$\+presentable} if
the functor $\Hom_\sC(S,{-})\:\sC\rarrow\Sets$ preserves
$\kappa$\+directed colimits (see~\cite[Definition~1.13(2)]{AR}).
 We will denote the class of all $\kappa$\+presentable objects of $\sC$
by $\sC_{<\kappa}\subset\sC$.
 The $\aleph_0$\+presentable objects are called {finitely presentable}
(see~\cite[Definition~1.1]{AR}); and the $\aleph_1$\+presentable objects
can be similarly called \emph{countably presentable}.

 This category-theoretic terminology is consistent with
the module-theoretic one.
 Given an associative ring $A$, an $A$\+module $M$ is 
$\kappa$\+presentable as an object of the category of left $A$\+modules
$A\Modl$ (in the sense of the definition above) if and only if it is
the cokernel of a morphism of free $A$\+modules with less than~$\kappa$
generators.

 Given an additive category $\sE$, we denote by $\Com(\sE)$
the category of (unbounded) cochain complexes in~$\sE$.
 Clearly, if $\kappa$\+directed colimits exist in $\sE$, then they
also exist in $\Com(\sE)$.
 The following lemma is standard.

\begin{lem} \label{complexes-presentability}
\textup{(a)} Let\/ $\sE$ be an additive category with directed colimits.
 Then any \emph{bounded} complex of finitely presentable objects of\/
$\sE$ is a finitely presentable object in\/ $\Com(\sE)$. \par
\textup{(b)} Let $\kappa$~be an \emph{uncountable} regular cardinal
and\/ $\sE$ be an additive category with $\kappa$\+directed colimits.
 Then any complex of $\kappa$\+presentable objects of\/ $\sE$ is
a $\kappa$\+presentable object in\/ $\Com(\sE)$.
\end{lem}

\begin{proof}
 Part~(a) holds, because directed colimits commute with finite limits
in the category of abelian groups.
 Part~(b) is true due to the fact that $\aleph_1$\+directed colimits
commute with countable limits of abelian groups.
\end{proof}

 A category $\sC$ with $\kappa$\+directed colimits is said to be
\emph{$\kappa$\+accessible} if there is a \emph{set} of
$\kappa$\+presentable objects $\sS\subset\sC$ such that
every object of $\sC$ is a $\kappa$\+directed colimit of objects
from~$\sS$ (see~\cite[Definition~2.1]{AR}).
 If this is the case, then the $\kappa$\+presentable objects of $\sC$
are precisely all the retracts of objects from~$\sS$.
 A $\kappa$\+accessible category $\sC$ is said to be \emph{locally
$\kappa$\+presentable} if all colimits exist in~$\sC$
(see~\cite[Definition~1.17 and Theorem~1.20]{AR}).
 The $\aleph_0$\+accessible categories are called \emph{finitely
accessible}~\cite[Remark~2.2(1)]{AR}, and the locally
$\aleph_0$\+presentable categories are called \emph{locally finitely
presentable}~\cite[Definition~1.9 and Theorem~1.11]{AR}.

 We denote by $A\Modl_\flat\subset A\Modl$ the full subcategory of
flat left $A$\+modules.

\begin{lem} \label{flat-modules-accessible}
 For any ring $A$ and any regular cardinal~$\kappa$, the category of
flat left $A$\+modules $A\Modl_\flat$ is $\kappa$\+accessible.
 The $\kappa$\+presentable objects of $A\Modl_\flat$ are precisely
all the $\kappa$\+presentable flat $A$\+modules (i.~e., the flat
$A$\+modules that are $\kappa$\+presentable in $A\Modl$).
\end{lem}

\begin{proof}
 The Govorov--Lazard description of flat $A$\+modules as the directed
colimits of finitely generated projective $A$\+modules~\cite{Gov},
\cite{Laz}, \cite[Corollary~2.22]{GT} means that the category
$A\Modl_\flat$ is finitely accessible, and the finitely generated
projective $A$\+modules are its finitely presentable objects.
 By~\cite[Theorem~2.11 and Example~2.13(1)]{AR}, it follows that
the category $A\Modl_\flat$ is $\kappa$\+presentable for all
regular cardinals~$\kappa$; and the argument in~\cite[proof of
Theorem~2.11\,(iv)\,$\Rightarrow$\,(i)]{AR} implies the desired
description of $\kappa$\+presentable objects.
 The point is that all flat $A$\+modules that are $\kappa$\+presentable
in $A\Modl$ are also $\kappa$\+presentable in $A\Modl_\flat$, since
the full subcategory $A\Modl_\flat$ is closed under directed colimits
in $A\Modl$.
 On the other hand, any directed colimit of $\kappa$\+presentable flat
$A$\+modules indexed by a directed poset of cardinality smaller
than~$\kappa$ is a $\kappa$\+presentable flat $A$\+module again.
\end{proof}

 The following lemma illustrates the significance of countably
presentable (i.~e., $\aleph_1$\+presentable) flat modules.

\begin{lem} \label{count-pres-flat-mods}
 Any countably presentable flat module over an associative ring $A$ is
a countable directed colimit of finitely generated free $A$\+modules.
 Consequently, the projective dimension of any countably presentable
flat module is less than or equal to\/~$1$.
\end{lem}

\begin{proof}
 See~\cite[Corollary~2.23]{GT}.
 The second assertion of the lemma follows from the first one
because the telescope construction of countable directed colimits
provides a two-term projective resolution of any countable directed
colimit of projective $A$\+modules.
\end{proof}

 In the rest of this section, we discuss several category-theoretic
constructions which will be used in this paper: the product,
the pseudopullback, the isomorpher, the inserter, the equifier,
and the diagram category.
 We recall results from the papers~\cite{Ulm,CR,RR,Pacc} concerning
$\kappa$\+accessibility of the categories produced by such constructions
and the descriptions of $\kappa$\+presentable objects in these
categories.

 In almost all the contexts, we will consider a regular
cardinal~$\kappa$ and a smaller infinite cardinal $\lambda<\kappa$
(so $\kappa$~has to be uncountable).
 In the applications in the main body of the paper, we will use
$\kappa=\aleph_1$ and $\lambda=\aleph_0$.

 A \emph{$\lambda$\+indexed chain} (of objects and morphisms) in
a category $\sC$ is a functor $\lambda\rarrow\sC$, where $\lambda$ is
considered as an ordinal, and this ordinal, viewed as an ordered set,
is considered as a small category.
 In other words, a $\lambda$\+indexed chain is a directed diagram
$(C_i\to C_j)_{0\le i<j<\lambda}$ in~$\sC$.

\begin{prop} \label{product-proposition}
 Let $\kappa$~be a regular cardinal and\/ $\Xi$ be a set of
cardinality smaller than~$\kappa$.
 Let $(\sK_\xi)_{\xi\in\Xi}$ be a family of $\kappa$\+accessible
categories indexed by the set\/~$\Xi$.
 Then the Cartesian product\/ $\sK=\prod_{\xi\in\Xi}\sK_\xi$ is also
a $\kappa$\+accessible category.
 The $\kappa$\+presentable objects of\/ $\sK$ are precisely all
the collections of objects $(S_\xi\in\sK_\xi)_{\xi\in\Xi}$ with
$S_\xi\in(\sK_\xi)_{<\kappa}$ for all $\xi\in\Xi$.
\end{prop}

\begin{proof}
 This is a corrected version of~\cite[proof of Proposition~2.67]{AR}.
 See~\cite[Proposition~2.1]{Pacc} for the details.
\end{proof}

 Let $\sK_1$, $\sK_2$, and $\sL$ be three categories, and let
$F_1\:\sK_1\rarrow\sL$ and $F_2\:\sK_2\rarrow\sL$ be two functors.
 The \emph{pseudopullback} $\sC$ of the pair of functors $F_1$, $F_2$
is defined as the category of triples $(K_1,K_2,\theta)$, where
$K_1\in\sK_1$ and $K_2\in\sK_2$ are two objects, and $\theta\:F_1(K_1)
\simeq F_2(K_2)$ is an isomorphism in~$\sL$.

\begin{thm} \label{pseudopullback-theorem}
 Let $\kappa$ be a regular cardinal and $\lambda<\kappa$ be a smaller
infinite cardinal.
 Let\/ $\sK_1$, $\sK_2$, and\/ $\sL$ be three $\kappa$\+accessible
categories where all $\lambda$\+indexed chains have colimits.
 Assume that two functors $F_1\:\sK_1\rarrow\sL$ and $F_2\:\sK_2\rarrow
\sL$ preserve $\kappa$\+directed colimits and colimits of
$\lambda$\+indexed chains.
 Assume further that the functors $F_1$ and $F_2$ take
$\kappa$\+presentable objects to $\kappa$\+presentable objects.
 Then the pseudopullback category\/ $\sC$ is $\kappa$\+accessible.
 The $\kappa$\+presentable objects of\/ $\sC$ are precisely all
the triples $(S_1,S_2,\theta)\in\sC$ with $S_1\in(\sK_1)_{<\kappa}$
and $S_2\in(\sK_2)_{<\kappa}$.
\end{thm}

\begin{proof}
 This is~\cite[Pseudopullback Theorem~2.2]{RR}, based on the argument
in~\cite[proof of Proposition~3.1]{CR}.
 The assertion essentially goes back to~\cite[Remark~3.2(I),
Theorem~3.8, Corollary~3.9, and Remark~3.11(II)]{Ulm}.
 See also~\cite[Corollary~5.1]{Pacc}.
\end{proof}

 Let $\sK$ and $\sL$ be two categories, and let $F_1$, $F_2\:\sK
\rightrightarrows\sL$ be two parallel functors.
 The \emph{isomorpher} $\sC$ of the pair of functors $F_1$, $F_2$ is
defined as the category of pairs $(K,\theta)$, where $K\in\sK$ is
an object and $\theta\:F_1(K)\simeq F_2(K)$ is an isomorphism in~$\sL$.

\begin{thm} \label{isomorpher-theorem}
 Let $\kappa$ be a regular cardinal and $\lambda<\kappa$ be a smaller
infinite cardinal.
 Let\/ $\sK$ and \/ $\sL$ be two $\kappa$\+accessible categories
where all $\lambda$\+indexed chains have colimits.
 Assume that two parallel functors $F_1$, $F_2\: \sK\rightrightarrows
\sL$ preserve $\kappa$\+directed colimits and colimits of
$\lambda$\+indexed chains.
 Assume further that the functors $F_1$ and $F_2$ take
$\kappa$\+presentable objects to $\kappa$\+presentable objects.
 Then the isomorpher category\/ $\sC$ is $\kappa$\+accessible.
 The $\kappa$\+presentable objects of\/ $\sC$ are precisely all
the pairs $(S,\theta)\in\sC$ with $S\in\sK_{<\kappa}$.
\end{thm}

\begin{proof}
 This is an equivalent reformulation of
Theorem~\ref{pseudopullback-theorem}; see~\cite[Remark~5.2]{Pacc}.
 The assertion can be also viewed as a particular case
of~\cite[Remark~3.2(I), Theorem~3.8, Corollary~3.9, and
Remark~3.11(II)]{Ulm}.
\end{proof}

 Let $\sK$ and $\sL$ be two categories, and $F$, $G\:\sK
\rightrightarrows\sL$ be two parallel functors.
 The \emph{inserter} $\sD$ of the pair of functors $F$, $G$ is
defined~\cite[Section~2.71]{AR} as the category of pairs $(K,\phi)$,
where $K\in\sK$ is an object and $\phi\:F(K)\rarrow G(K)$ is
a morphism in~$\sL$.
 
\begin{thm} \label{inserter-theorem}
 Let $\kappa$ be a regular cardinal and $\lambda<\kappa$ be a smaller
infinite cardinal.
 Let\/ $\sK$ and \/ $\sL$ be two $\kappa$\+accessible categories
where all $\lambda$\+indexed chains have colimits.
 Assume that two parallel functors $F$, $G\: \sK\rightrightarrows
\sL$ preserve $\kappa$\+directed colimits and colimits of
$\lambda$\+indexed chains.
 Assume further that the functor $F$ takes $\kappa$\+presentable objects
to $\kappa$\+presentable objects.
 Then the inserter category\/ $\sD$ is $\kappa$\+accessible.
 The $\kappa$\+presentable objects of\/ $\sD$ are precisely all
the pairs $(S,\phi)\in\sD$ with $S\in\sK_{<\kappa}$.
\end{thm}

\begin{proof}
 This is~\cite[Theorem~3.8, Corollary~3.9, and Remark~3.11(II)]{Ulm}
or~\cite[Theorem~4.1]{Pacc}.
\end{proof}

 Let $\sD$ and $\sM$ be two categories, $F$, $G\:\sD\rightrightarrows
\sM$ be two parallel functors, and $\phi$, $\psi\:F\rightrightarrows G$
be two parallel natural transformations.
 The \emph{equifier} $\sE$ of the pair of natural transformations
$\phi$,~$\psi$ is defined~\cite[Lemma~2.76]{AR} as the full subcategory
in $\sD$ consisting of all objects $E\in\sD$ for which the two morphisms
$\phi_E\:F(E)\rarrow G(E)$ and $\psi_E\:F(E)\rarrow G(E)$ are equal
to each other, i.~e., $\phi_E=\psi_E$.

\begin{thm} \label{equifier-theorem}
 Let $\kappa$ be a regular cardinal and $\lambda<\kappa$ be a smaller
infinite cardinal.
 Let\/ $\sD$ and \/ $\sM$ be two $\kappa$\+accessible categories
where all $\lambda$\+indexed chains have colimits.
 Assume that two parallel functors $F$, $G\: \sD\rightrightarrows
\sM$ preserve $\kappa$\+directed colimits and colimits of
$\lambda$\+indexed chains.
 Assume further that the functor $F$ takes $\kappa$\+presentable objects
to $\kappa$\+presentable objects.
 Let $\phi$, $\psi\:F\rightrightarrows G$ be two parallel natural
transformations.
 Then the equifier category\/ $\sE$ is $\kappa$\+accessible.
 The $\kappa$\+presentable objects of\/ $\sE$ are precisely all
the objects of\/ $\sE$ that are $\kappa$\+presentable as objects
of\/~$\sD$.
\end{thm}

\begin{proof}
 This is~\cite[Theorem~3.8, Corollary~3.9, and Remark~3.11(II)]{Ulm}
or~\cite[Theorem~3.1]{Pacc}.
\end{proof}

 Let $k$~be a commutative ring.
 The notion of a \emph{$\kappa$\+presented $k$\+linear category}
spelled out in~\cite[Section~6]{Pacc} represents an intuitively clear
idea of a quiver with less than~$\kappa$ vertices, less than~$\kappa$
arrows, and less than~$\kappa$ relations between $k$\+linear
combinations of various iterated compositions of the arrows.
 We refer to~\cite{Pacc} for the details.
 Given a small $k$\+linear category $D$ and a $k$\+linear category
$\sK$, we denote by $\Fun_k(D,\sK)$ the category of $k$\+linear
functors $D\rarrow\sK$.

\begin{thm} \label{diagram-theorem}
 Let $\kappa$ be a regular cardinal and $\lambda<\kappa$ be a smaller
infinite cardinal.
 Let\/ $\sC$ be a $\kappa$\+accessible $k$\+linear category where all
$\lambda$\+indexed chains have colimits, and let $D$ be
a $\kappa$\+presented $k$\+linear category.
 Then the diagram category\/ $\Fun_k(D,\sC)$ is $\kappa$\+accessible.
 The $\kappa$\+presentable objects of\/ $\Fun_k(D,\sC)$ are precisely
all the $k$\+linear functors $D\rarrow\sC_{<\kappa}$.
\end{thm}

\begin{proof}
 This is~\cite[Theorem~6.2]{Pacc}.
\end{proof}

\Section{Preliminaries on Corings and Comodules}
\label{comod-preliminaries-secn}

 The following definitions seem to go back to Sweedler's
paper~\cite{Swe}.
 A detailed exposition can be found in the book~\cite{BW}.
 The book~\cite{Psemi} and the survey paper~\cite[Sections~2.5
and~3.4]{Prev} can be used as further references.

 A \emph{coring} over an associative ring $A$ is a comonoid object in
the monoidal category of $A$\+$A$\+bimodules with respect to
the tensor product over~$A$.
 In other words, a coring $\C$ over $A$ is an $A$\+$A$\+bimodule
endowed with $A$\+$A$\+bimodule maps of \emph{comultiplication}
$\mu\:\C\rarrow\C\ot_A\C$ and \emph{counit} $\epsilon\:\C\rarrow A$
satisfying the usual coassociativity and counitality axioms.
 Specifically, the two compositions $(\mu\ot\id_\C)\circ\mu$ and
$(\id_\C\ot\mu)\circ\mu$ must be equal to each other,
$$
 \C\rarrow\C\ot_A\C\rightrightarrows\C\ot_A\C\ot_A\C,
$$
and both the compositions $(\epsilon\ot\id_\C)\circ\mu$ and
$(\id_\C\ot\epsilon)\circ\mu$ must be equal to
the identity map~$\id_\C$,
$$
 \C\rarrow\C\ot_A\C\rightrightarrows\C.
$$

 The category of left $A$\+modules $A\Modl$ is a left module category
(with respect to the tensor product operation) over the monoidal
category of $A$\+$A$\+bimodules $A\Bimod A$.
 A \emph{left\/ $\C$\+comodule} is a comodule object in the module
category $A\Modl$ over the comonoid object $\C$ in the monoidal
category $A\Bimod A$.
 In other words, a left comodule $\M$ over $\C$ is a left $A$\+module
endowed with a left $A$\+module map of \emph{left coaction}
$\nu\:\M\rarrow\C\ot_A\M$ satisfying the usual coassociativity and
counitality axioms together with the comultiplication map~$\mu$ and
the counit map~$\epsilon$ of the coring~$\C$.
 Specifically, the two compositions $(\mu\ot\id_\M)\circ\nu$ and
$(\id_\C\ot\nu)\circ\nu$ must be equal to each other,
$$
 \M\rarrow\C\ot_A\M\rightrightarrows\C\ot_A\C\ot_A\M,
$$
and the composition $(\epsilon\ot\id_\M)\circ\nu$ must be equal to
the identity map~$\id_\M$,
$$
 \M\rarrow\C\ot_A\M\rarrow\M.
$$

 We denote the additive category of left $\C$\+comodules by
$\C\Comodl$, and use the notation $\Hom_\C({-},{-})$ for the groups
of morphisms in $\C\Comodl$.
 A $\C$\+comodule is said to be \emph{$A$\+flat} if it is flat as
an $A$\+module.
 The full subcategory of $A$\+flat left $\C$\+comodules is denoted
by $\C\Comodl_{A\dfl}\subset\C\Comodl$.

\begin{lem} \label{comodule-category-lemma}
 Let\/ $\C$ be a coring over a ring~$A$. \par
\textup{(a)} All colimits exist in the additive category of
left\/ $\C$\+comodules.
 The forgetful functor\/ $\C\Comodl\rarrow A\Modl$ preserves colimits
(in particular, coproducts). \par
\textup{(b)} If\/ $\C$ is a flat right $A$\+module, then the category\/
$\C\Comodl$ is a Grothendieck abelian category, and the forgetful
functor\/ $\C\Comodl\rarrow A\Modl$ is exact. \par
\textup{(c)} Conversely, if the category\/ $\C\Comodl$ is abelian
\emph{and} the forgetful functor\/ $\C\Comodl\rarrow A\Modl$ is exact,
then\/ $\C$ is a flat right $A$\+module. \par
\textup{(d)} There exists a coring\/ $\C$ such that the category\/
$\C\Comodl$ is \emph{not} abelian.
\end{lem}

\begin{proof}
 Part~(a) follows from the fact that the tensor product functor
$\C\ot_A{-}$ preserves colimits in $A\Modl$.
 Parts~(b\+-c) are covered by~\cite[Sections~18.6, 18.14
and~18.16]{BW}; see also~\cite[Proposition~2.12(a)]{Prev}.
 A counterexample for part~(d) can be found
in~\cite[Example~C.1.1]{Pcosh}.
\end{proof}

 When $\C$ is not a flat right $A$\+module, exact sequences of
arbitrary left $\C$\+comodules are not well-behaved.
 But there is still a well-behaved notion of a \emph{short exact
sequence of $A$\+flat left\/ $\C$\+comodules}, defined as
a short sequence of $A$\+flat left $\C$\+comodules that is exact
in $A\Modl$.
 We refer to the survey~\cite{Bueh} for background material on
exact categories in the sense of Quillen.

\begin{lem} \label{flat-comodules-exact-category}
 The additive category of $A$\+flat $\C$\+comodules\/
$\C\Comodl_{A\dfl}$ endowed with the class of short exact sequences
defined above is an exact category.
\end{lem}

\begin{proof}
 Essentially, it suffices to check that the kernel of any surjective
morphism of $A$\+flat $\C$\+comodules exists in $\C\Comodl_{A\dfl}$
and agrees with the kernel of the same morphism computed in
the category of flat $A$\+modules.
 For this purpose, one needs to use the fact that the functors
$\C\ot_A{-}$ and $\C\ot_A\C\ot_A{-}$ take short exact sequences
of flat $A$\+modules to short exact sequences of $A$\+modules.
\end{proof}

 More generally, an exact sequence of left $A$\+modules is called
\emph{pure} if it stays exact after the functor $N\ot_A{-}$ is
applied to it, for every right $A$\+module~$N$.
 An \emph{$A$\+pure short exact sequence of\/ $\C$\+comodules} is
a short sequence of $\C$\+comodules that is pure exact as
a short sequence of $A$\+modules.
 The category of left $\C$\+comodules $\C\Comodl$ with $A$\+pure short
exact sequences is also an exact category.

 A left $\C$\+comodule is said to be \emph{$A$\+countably presentable}
if it is countably presentable as an $A$\+module.
 Similarly, a $\C$\+comodule is said to be \emph{$A$\+finitely
presentable} if it is finitely presentable as an $A$\+module.

\begin{lem} \label{comodule-presentability-lemma}
 Let\/ $\C$ be a coring over a ring~$A$.  Then \par
\textup{(a)} all $A$\+finitely presentable\/ $\C$\+comodules are
finitely presentable as objects of the additive category\/ $\C\Comodl$; 
\par
\textup{(b)} all \emph{bounded} complexes of $A$\+finitely presentable\/
$\C$\+comodules are finitely presentable as objects of the additive
category of complexes of\/ $\C$\+comodules; \par
\textup{(c)} all $A$\+countably presentable\/ $\C$\+comodules are
countably presentable (i.~e., $\aleph_1$\+presentable) as objects of
the additive category\/ $\C\Comodl$; \par
\textup{(d)} all complexes of $A$\+countably presentable\/
$\C$\+comodules are countably presentable as objects of the additive
category of complexes of\/ $\C$\+comodules.
\end{lem}

\begin{proof}
 Part~(b) follows from part~(a) according to
Lemma~\ref{complexes-presentability}(a); and part~(d) follows from
part~(c) in view of Lemma~\ref{complexes-presentability}(b).
 Parts~(a) and~(c) follow from the facts that, for any two left
$\C$\+comodules $\LL$ and $\M$, the abelian group $\Hom_\C(\LL,\M)$
can be computed as the kernel of (the difference of) a natural pair
of maps
$$
 \Hom_A(\LL,\M)\,\rightrightarrows\,\Hom_A(\LL,\>\C\ot_A\M),
$$
and the functor $\C\ot_A{-}$ preserves colimits.
\end{proof}

 For the converse assertions to
Lemma~\ref{comodule-presentability-lemma}(c\+-d),
see Remarks~\ref{arbitrary-comodules-aleph1-loc-presentable}
and~\ref{arbitrary-complexes-of-comodules} below.
 It will be also explained there that the additive categories of
left $\C$\+comodules $\C\Comodl$ and complexes of left
$\C$\+comodules $\Com(\C\Comodl)$ are locally $\aleph_1$\+presentable.

\begin{lem} \label{finite-weak-dimension-lemma}
 Let\/ $\C$ be a coring over a ring~$A$.
 Assume that\/ $\C$ is a flat left and right $A$\+module, and the ring
$A$ has finite weak global dimension (or in other words, finite\/
$\Tor$\+dimension).
 Then every\/ $\C$\+comodule is a quotient comodule of an $A$\+flat\/
$\C$\+comodule.
\end{lem}

\begin{proof}
 This is~\cite[Lemma~1.1.3]{Psemi}.
 For a discussion of this construction in a module-theoretic context,
see~\cite[Section~2.2]{Pctrl}.
\end{proof}

\begin{ex} \label{kontsevich-rosenberg}
 This example is due to Kontsevich and
Rosenberg~\cite[Section~2]{KR}, \cite{KR2}.
 Let $X$ be a quasi-compact semi-separated scheme, and let
$X=\bigcup_\alpha U_\alpha$ be a finite covering of $X$ by affine
open subschemes.
 Let $U$ denote the disjoint union $U=\coprod_\alpha U_\alpha$; so
$U$ is an affine scheme and there is a natural surjective flat affine
morphism of schemes $U\rarrow X$.
 Let $A=\cO(U)$ and $\C=\cO(U\times_XU)$ denote the rings of functions
on the affine schemes $U$ and $U\times_XU$.
 Then the two projections $p_1$, $p_2\:U\times_XU\rightrightarrows U$
induce two ring homomorphisms $A\rightrightarrows\C$, endowing $\C$
with two $A$\+module structures (which we view as the ``left'' and
the ``right'' one).

 The diagonal morphism $\delta\:U\rarrow U\times_XU$ induces a ring
homomorphism $\delta\spcheck\:\C\rarrow A$.
 Furthermore, we have $\C\ot_A\C=\cO((U\times_XU)\times_U(U\times_XU))
=\cO(U\times_XU\times_XU)$.
 The projection $q\:U\times_XU\times_XU\rarrow U\times_XU$ dropping
the second (inner) coordinate induces a ring homomorphism
$q\spcheck\:\C\rarrow\C\ot_A\C$.
 The maps $\epsilon=\delta\spcheck\:\C\rarrow A$ and $\mu=q\spcheck\:
\C\rarrow\C\ot_A\C$ that we have constructed endow
the $A$\+$A$\+bimodule $\C$ with a coring structure
(over the commutative ring~$A$).

 For any quasi-coherent sheaf $N$ on $X$, the inverse image of $N$
with respect to the morphism $U\rarrow X$ produces a quasi-coherent
sheaf $M$ on $U$, that is an $A$\+module~$\M$.
 The inverse images of $M$ with respect to the two projections
$U\times_XU\rightrightarrows U$ are naturally isomorphic as
quasi-coherent sheaves on~$U\times_XU$, that is $p_1^*M\simeq p_2^*M$.
 The corresponding bimodules of sections are $\M\ot_A\C$ and
$\C\ot_A\M$; so we have a natural isomorphism of $A$\+$A$\+bimodules
$\M\ot_A\C\simeq\C\ot_A\M$.
 Precomposing this isomorphism with the map $\M\rarrow\M\ot_A\C$
induced by the ``left'' map $A\rarrow\C$ (induced by~$p_1$), we obtain
a map $\nu\:\M\rarrow\C\ot_A\M$, endowing $\M$ with a left
$\C$\+comodule structure.
 The assignment $N\longmapsto\M$ is an equivalence of abelian categories
$X\qcoh\simeq\C\Comodl$.
 Here $X\qcoh$ denotes the abelian category of quasi-coherent sheaves
on~$X$.
\end{ex}

\begin{rem} \label{flat-sheaves-and-comodules}
 A quasi-coherent sheaf $F$ on a scheme $X$ is said to be \emph{flat} if
the tensor product functor $F\ot_{\cO_X}{-}\,\:X\qcoh\rarrow X\qcoh$ is
exact, or equivalently, the $\cO(U)$\+module $F(U)$ is flat for every
affine open subscheme $U\subset X$.
 It suffices to check the latter condition for the affine open
subschemes $U=U_\alpha$ from a given affine open covering
$X=\bigcup_\alpha U_\alpha$ of the scheme~$X$.

 Now let $X$ be a quasi-compact semi-separated scheme.
 Then it follows that, in the context of
Example~\ref{kontsevich-rosenberg}, a quasi-coherent sheaf on $X$ is
flat if and only if the corresponding $\C$\+comodule is $A$\+flat.

 By~\cite[Section~2.4]{M-n} or~\cite[Lemma~A.1]{EP}, any quasi-coherent
sheaf on a quasi-compact semi-separated scheme $X$ is a quotient sheaf
of a flat quasi-coherent sheaf.
 Therefore, the corings $\C$ appearing in
Example~\ref{kontsevich-rosenberg} also have the property that any
$\C$\+comodule is a quotient comodule of an $A$\+flat $\C$\+comodule.
\end{rem}

\Section{$A$-Flat $\C$-Comodules as Directed Colimits}
\label{flat-comodules-as-directed-colimits-secn}

 The following theorem describing $A$\+flat $\C$\+comodules is
the first main result of this paper.
 In view of Example~\ref{kontsevich-rosenberg} and
Remark~\ref{flat-sheaves-and-comodules}, it is a generalization
of the assertion of~\cite[Theorem~2.4]{PS6} for quasi-compact
semi-separated schemes.

\begin{thm} \label{flat-comodules-as-directed-colimits-theorem}
 Let\/ $\C$ be a coring over an associative ring~$A$.
 Then the category\/ $\C\Comodl_{A\dfl}$ of $A$\+flat left\/
$\C$\+comodules is\/ $\aleph_1$\+accessible.
 The\/ $\aleph_1$\+presentable objects of\/ $\C\Comodl_{A\dfl}$ are
precisely all the $A$\+countably presentable $A$\+flat\/
$\C$\+comodules.
 Consequently, every $A$\+flat\/ $\C$\+comodule is
an\/ $\aleph_1$\+directed colimit of $A$\+countably presentable
$A$\+flat\/ $\C$\+comodules.
\end{thm}

\begin{proof}
 All $A$\+countably presentable $\C$\+comodules are
$\aleph_1$\+presentable as objects of $\C\Comodl$ by
Lemma~\ref{comodule-presentability-lemma}(c).
 Since the full subcategory $\C\Comodl_{A\dfl}$ is closed under
directed colimits in $\C\Comodl$, it follows that all
$A$\+countably presentable $A$\+flat $\C$\+comodules are
$\aleph_1$\+presentable in $\C\Comodl_{A\dfl}$.

 The full assertion of the theorem is obtained by applying
Theorems~\ref{inserter-theorem} and~\ref{equifier-theorem} for
$\kappa=\aleph_1$ and $\lambda=\aleph_0$.
 One starts with the category $\sK=A\Modl_\flat$, the category
$\sL=A\Modl$, and the pair of parallel functors $F$, $G\:\sK
\rightrightarrows\sL$ given by the rules $F(M)=M$ and $G(M)=\C\ot_AM$
for all $M\in A\Modl_\flat$.
 Then the inserter category $\sD$ is the category of all
``noncoassociative, noncounital'' $A$\+flat left $\C$\+comodules,
i.~e., flat left $A$\+modules $\D$ endowed with a left $A$\+module
morphism $\nu\:\D\rarrow\C\ot_A\D$.

 The category $\sK=A\Modl_\flat$ is $\aleph_1$\+accessible by
Lemma~\ref{flat-modules-accessible}; and the abelian category
$A\Modl$ is locally finitely presentable, hence locally
$\kappa$\+presentable for every regular cardinal~$\kappa$
\,\cite[Remark~1.20]{AR}.
 Theorem~\ref{inserter-theorem} is applicable (notice that the functor
$F$ takes $\aleph_1$\+presentable objects to $\aleph_1$\+presentable
objects, as it must; while the functor $G$ does not, and it does not
have to).
 The theorem tells us that the category $\sD$ is $\aleph_1$\+accessible,
and provides a description of the $\aleph_1$\+presentable objects
of~$\sD$.

 Now let $\sM$ be the Cartesian product of categories
$\sM=A\Modl\times A\Modl_\flat$.
 The functor $F\:\sD\rarrow\sM$ assigns to a noncoassociative, 
noncounital $A$\+flat $\C$\+comodule $\D$ the pair of $A$\+modules
$(\D,\D)$.
 The functor $G\:\sD\rarrow\sM$ assigns to $\D$ the pair of
$A$\+modules $(\C\ot_A\C\ot_A\D,\>\D)$.
 The natural transformation $\phi\:F\rarrow G$ assigns to $\D$
the pair of compositions $(\mu\ot\id_\D)\circ\nu\:\D\rarrow\C\ot_A\D
\rarrow\C\ot_A\C\ot_A\D$ and $(\epsilon\ot\id_D)\circ\nu\:\D\rarrow
\C\ot_A\D\rarrow\D$.
 The natural transformation $\psi\:F\rarrow G$ assigns to $\D$
the pair of maps consisting of the composition
$(\id_\C\ot\nu)\circ\nu\:\D\rarrow\C\ot_A\D\rarrow\C\ot_A\C\ot_A\D$
and the identity map $\id_\D\:\D\rarrow\D$.
 Then the equifier category $\sE$ is the category of coassociative,
counital $A$\+flat left $\C$\+comodules $\C\Comodl_{A\dfl}$.

 Proposition~\ref{product-proposition} tells us that the category $\sM$
is $\aleph_1$\+accessible, and describes its full subcategory of
$\aleph_1$\+presentable objects.
 Theorem~\ref{equifier-theorem} is applicable (once again,
the functor $F$ takes $\aleph_1$\+presentable objects to
$\aleph_1$\+presentable objects, while the functor $G$ does not).
 The theorem tells us that the category $\sE=\C\Comodl_{A\dfl}$ is
$\aleph_1$\+accessible, and provides the desired description of
the $\aleph_1$\+presentable objects of~$\sE$.
\end{proof}

\begin{rem} \label{arbitrary-comodules-aleph1-loc-presentable}
 Replacing the original category $\sK=A\Modl_\flat $ by $\sK=A\Modl$
(and the category $\sM=A\Modl\times A\Modl_\flat$ by
$\sM=A\Modl\times A\Modl$) in the proof of
Theorem~\ref{flat-comodules-as-directed-colimits-theorem},
one obtains a proof of the assertion that the category of all comodules
$\C\Comodl$ is $\aleph_1$\+accessible for any coring~$\C$.
 In fact, all colimits exist in $\C\Comodl$ by
Lemma~\ref{comodule-category-lemma}(a); so the additive category
$\C\Comodl$ is even locally $\aleph_1$\+presentable.
 Furthermore, Theorems~\ref{inserter-theorem} and~\ref{equifier-theorem}
tell us that the $\aleph_1$\+presentable objects of $\C\Comodl$ are
precisely all the $A$\+countably presentable $\C$\+comodules.
 So any $\C$\+comodule is an $\aleph_1$\+directed colimit of
$A$\+countably presentable $\C$\+comodules.
 Thus we obtain a much more precise version of~\cite[Theorem~10]{Por}.
\end{rem}

 Let us also provide a description of arbitrary complexes of
$A$\+flat $\C$\+comodules as directed colimits.

\begin{prop} \label{complexes-of-flat-comodules-prop}
 Let\/ $\C$ be a coring over an associative ring~$A$.
 Then the category\/ $\Com(\C\Comodl_{A\dfl})$ of complexes of
$A$\+flat left\/ $\C$\+comodules is\/ $\aleph_1$\+accessible.
 The\/ $\aleph_1$\+presentable objects of\/ $\Com(\C\Comodl_{A\dfl})$
are precisely all the complexes of $A$\+countably presentable
$A$\+flat\/ $\C$\+comodules.
 Consequently, every complex of $A$\+flat\/ $\C$\+comodules is
an\/ $\aleph_1$\+directed colimit of complexes of
$A$\+countably presentable $A$\+flat\/ $\C$\+comodules.
\end{prop}

\begin{proof}
 All complexes of $A$\+countably presentable $\C$\+comodules are
$\aleph_1$\+presentable as objects of $\Com(\C\Comodl)$ by
Lemma~\ref{comodule-presentability-lemma}(d).
 Since the full subcategory $\Com(\C\Comodl_{A\dfl})$ is closed under
directed colimits in $\Com(\C\Comodl)$, it follows that all complexes
of $A$\+countably presentable $A$\+flat $\C$\+comodules are
$\aleph_1$\+presentable in $\Com(\C\Comodl_{A\dfl})$.
{\hbadness=1800\par}

 The full assertion of the proposition is obtained by combining
the results of
Theorems~\ref{flat-comodules-as-directed-colimits-theorem}
and~\ref{diagram-theorem}.
 All one needs to do is to produce a suitable $\aleph_1$\+presented
$\boZ$\+linear category $D$ describing complexes in additive
categories.
 This simple construction can be found
in~\cite[proof of Corollary~10.4]{Pacc}.
\end{proof}

\begin{rem} \label{arbitrary-complexes-of-comodules}
 Similarly to the proof of
Proposition~\ref{complexes-of-flat-comodules-prop}, it follows from
Remark~\ref{arbitrary-comodules-aleph1-loc-presentable} that
the category $\Com(\C\Comodl)$ of arbitrary complexes of $\C$\+comodules
is locally $\aleph_1$\+presentable.
 Furthermore, the $\aleph_1$\+presentable objects of $\Com(\C\Comodl)$
are precisely all the complexes of $A$\+countably presentable
$\C$\+comodules.
 So any complex of $\C$\+comodules in an $\aleph_1$\+directed colimit
of complexes of $A$\+countably presentable $\C$\+comodules.
 Alternatively, as the categories involved are locally presentable
rather than only accessible, it suffices to use a suitable
additive version of~\cite[Theorem~1.2]{Hen} instead of
Theorem~\ref{diagram-theorem} in this proof.
\end{rem}

\Section{Exact Sequences of $A$-Flat $\C$-Comodules
as Directed Colimits} \label{comodule-exact-sequences-secn}

 We start with three module-theoretic lemmas.

\begin{lem} \label{countable-flat-coherence}
 Let $A$ be an associative ring.
 Then the kernel of any surjective morphism from a countably presentable
flat $A$\+module to a countably presentable flat $A$\+module is
a countably presentable (flat) $A$\+module.
\end{lem}

\begin{proof}
 Notice that the kernel of a surjective morphism from a countably
presentable module to a countably presentable module need not be
countably presentable in general (unless the ring is countably
coherent).
 The lemma claims that for \emph{flat} modules it is.

 Let $F$ be a countably presentable flat $A$\+module.
 By Lemma~\ref{count-pres-flat-mods}, there exists a countable directed
diagram of finitely generated projective $A$\+modules $P_0\rarrow P_1
\rarrow P_2\rarrow\dotsb$ such that $F=\varinjlim_{n\ge0}P_n$.
 Hence we have a short exact sequence of $A$\+modules $0\rarrow
\bigoplus_{n=0}^\infty P_n\rarrow\bigoplus_{n=0}^\infty P_n\rarrow F
\rarrow0$.
 Put $P'=P''=\bigoplus_{n=0}^\infty P_n$.

 Now let $0\rarrow H\rarrow G\rarrow F\rarrow0$ be a short exact
sequence of $A$\+modules.
 Since the $A$\+module $P'$ is projective, the $A$\+module morphism
$P'\rarrow F$ lifts over the surjective $A$\+module morphism
$G\rarrow F$.
 We obtain a commutative diagram of a morphism of short exact sequences
acting by the identity on the rightmost terms
$$
 \xymatrix{
  0 \ar[r] & P'' \ar[r] \ar[d] & P' \ar[r] \ar[d] & F \ar[r] \ar@{=}[d]
  & 0 \\
  0 \ar[r] & H \ar[r] & G \ar[r] & F \ar[r] & 0
 }
$$

 Consequently, there is a short exact sequence of $A$\+modules
$0\rarrow P''\rarrow H\oplus P'\rarrow G\rarrow0$.
 Now we notice that the $A$\+module $P''$ is countably presentable.
 If the $A$\+module $G$ is countably presentable as well, then it
follows that the $A$\+module $H\oplus P'$ is countably presentable
(as the class of countably presentable $A$\+modules is closed under
extensions).
 Consequently, the $A$\+module $H$ is countably presentable as
a direct summand of a countably presentable $A$\+module.
\end{proof}

 For a generalization of Lemma~\ref{countable-flat-coherence} to
arbitrary regular cardinals~$\kappa$, see~\cite[Corollary~10.12]{Pacc}
or~\cite[Corollary~4.7]{Plce}.

\begin{lem} \label{three-term-complexes-of-modules}
 Let $A$ be an associative ring.
 Then the abelian category of three-term complexes of left $A$\+modules
$C^1\rarrow C^2\rarrow C^3$ is locally finitely presentable, and
consequently, locally\/ $\aleph_1$\+presentable.
 The\/ $\aleph_1$\+presentable objects of this category are precisely
all the three-term complexes of countably presentable $A$\+modules.
\end{lem}

\begin{proof}
 The assertion holds for any regular cardinal~$\kappa$ in lieu
of~$\aleph_1$.
 See~\cite[Proposition~1.1(c)]{Plce} for a much more general result.
\end{proof}

\begin{lem} \label{short-exact-of-flat-modules-accessible}
 For any ring $A$, the category of short exact sequences of flat
$A$\+modules is\/ $\aleph_1$\+accessible.
 The\/ $\aleph_1$\+presentable objects of this category are precisely
all the short exact sequences of countably presentable flat
$A$\+modules.
\end{lem}

\begin{proof}
 Once again, the assertion holds for any regular cardinal~$\kappa$.
 See~\cite[Corollary~10.13]{Pacc} or~\cite[Propositions~1.1(c)
and~4.6]{Plce}.
\end{proof}

 Recall from the discussion in Section~\ref{comod-preliminaries-secn}
that the category $\C\Comodl$ need not be abelian for a coring $\C$
over an associative ring $A$ in general; but one can still speak
about \emph{short exact sequences of $A$\+flat $\C$\+comodules}.
 The following proposition describes such short exact sequences.

\begin{prop} \label{short-exact-sequences-of-flat-comodules}
 Let\/ $\C$ be a coring over an associative ring~$A$.
 Then the category of short exact sequences of $A$\+flat\/
$\C$\+comodules is\/ $\aleph_1$\+accessible.
 The\/ $\aleph_1$\+presentable objects of this category are precisely
all the short exact sequences of $A$\+countably presentable $A$\+flat\/
$\C$\+comodules.
 Consequently, every short exact sequence of $A$\+flat\/ $\C$\+comodules
is an\/ $\aleph_1$\+directed colimit of short exact sequences of
$A$\+countably presentable $A$\+flat\/ $\C$\+comodules.
\end{prop}

\begin{proof}
 Similar to the proof of
Theorem~\ref{flat-comodules-as-directed-colimits-theorem}, with
the following changes.
 Use the category of short exact sequences of flat $A$\+modules
(from Lemma~\ref{short-exact-of-flat-modules-accessible}) in
the role of $\sK$, the category of three-term complexes of
$A$\+modules (from Lemma~\ref{three-term-complexes-of-modules})
in the role of $\sL$, and put $\sM=\sL\times\sK$.
 The functors and natural transformations are given by the same
formulas as in the proof of
Theorem~\ref{flat-comodules-as-directed-colimits-theorem}, applied
to short exact sequences or three-term complexes instead of (co)modules.
\end{proof}

 The assertion of
Proposition~\ref{short-exact-sequences-of-flat-comodules} (together
with Theorem~\ref{flat-comodules-as-directed-colimits-theorem})
can be restated by saying that the exact category of
$A$\+flat $\C$\+comodules $\C\Comodl_{A\dfl}$ from
Lemma~\ref{flat-comodules-exact-category} is
a \emph{locally\/ $\aleph_1$\+coherent exact category} in
the sense of~\cite[Section~1]{Plce}.

 We refer to~\cite[Section~10]{Bueh} for the definition of an acyclic
complex in an exact category.
 Acyclic complexes in the exact category $\C\Comodl_{A\dfl}$ are called
\emph{$A$\+pure acyclic complexes of $A$\+flat\/ $\C$\+comodules}.
 Equivalently, the $A$\+pure acyclic complexes of $A$\+flat
$\C$\+comodules are the complexes of $\C$\+comodules whose underlying
complexes of $A$\+modules are acyclic with flat $A$\+modules
of cocycles.

\begin{cor} \label{comods-pure-acycl-complexes-as-aleph1-dir-colims}
 Let\/ $\C$ be a coring over an associative ring~$A$.
 Then the category of $A$\+pure acyclic complexes of $A$\+flat\/
$\C$\+comodules is\/ $\aleph_1$\+accessible.
 The\/ $\aleph_1$\+presentable objects of this category are precisely
all the $A$\+pure acyclic complexes of $A$\+countably presentable
$A$\+flat\/ $\C$\+comodules.
 Consequently, every $A$\+pure acyclic complex of $A$\+flat\/
$\C$\+comodules is an\/ $\aleph_1$\+directed colimit of $A$\+pure
acyclic complexes of $A$\+countably presentable $A$\+flat\/
$\C$\+comodules.
\end{cor}

\begin{proof}
 This is provable similarly to
Proposition~\ref{short-exact-sequences-of-flat-comodules}
(using~\cite[Corollary~10.14]{Pacc} instead of
Lemma~\ref{short-exact-of-flat-modules-accessible}, and
the obvious version of Lemma~\ref{three-term-complexes-of-modules}
for unbounded complexes).
 This is akin to the argument used in~\cite[proof of
Theorem~4.2]{PS6}.
 But we prefer to spell out an argument deducing the corollary from
the proposition.

 The argument is similar to~\cite[proof of Corollary~10.14]{Pacc}.
 An $A$\+pure acyclic complex of $A$\+flat $\C$\+comodules $\F^\bu$ is
the same thing as a collection of short exact sequences of
$A$\+flat $\C$\+comodules $0\rarrow\G^n\rarrow\F^n\rarrow\HH^n\rarrow0$
together with a collection of isomorphisms of $\C$\+comodules
$\HH^n\simeq\G^{n+1}$, \,$n\in\boZ$.
 Therefore, the category of $A$\+pure acyclic complexes of $A$\+flat
$\C$\+comodules can be constructed from the category of short exact
sequences of $A$\+flat $\C$\+comodules using Cartesian products
(as in Proposition~\ref{product-proposition}) and the isomorpher
construction (as in Theorem~\ref{isomorpher-theorem}).

 Specifically, take $\kappa=\aleph_1$ and $\lambda=\aleph_0$.
 Denote by $\sK$ the Cartesian product of copies of the category
of short exact sequences of $A$\+flat $\C$\+comodules
(from Proposition~\ref{short-exact-sequences-of-flat-comodules}),
indexed by all the integers $n\in\boZ$.
 Denote by $\sL$ the Cartesian product of copies of the category
of $A$\+flat $\C$\+comodules $\C\Comodl_{A\dfl}$, also indexed by
$n\in\boZ$.
 Let $F_1\:\sK\rarrow\sL$ be the functor taking a collection of short
exact sequences $(0\to\G^n\to\F^n\to\HH^n\to0)_{n\in\boZ}$ to
the collection of flat comodules $(\HH^n)_{n\in\boZ}$, and let
$F_2\:\sK\rarrow\sL$ be the functor taking the same collection of
short exact sequences to the collection of flat comodules
$(\G^{n+1})_{n\in\boZ}$.
 Then the isomorpher category $\sC$ is the category of $A$\+pure
acyclic complexes of $A$\+flat $\C$\+comodules.

 Theorem~\ref{flat-comodules-as-directed-colimits-theorem}
and Proposition~\ref{short-exact-sequences-of-flat-comodules}
together with Proposition~\ref{product-proposition} tell us that
the categories $\sK$ and $\sL$ are $\aleph_1$\+accessible, and
describe their full subcategories of $\aleph_1$\+presentable objects.
 In view of these descriptions, Theorem~\ref{isomorpher-theorem} is
applicable, telling us that the isomorpher category $\sC$ is
$\aleph_1$\+accessible, and describing its full subcategory of
$\aleph_1$\+presentable objects.
\end{proof}

\Section{Cotorsion Periodicity for Comodules}
\label{comodule-cotorsion-periodicity-secn}

 In this section we deduce certain conditional results, based on
the techniques of~\cite[Section~8]{PS6} and generalizing some results
from~\cite[Section~9]{PS6}.

 Let $\C$ be a coring over a ring $A$ such that $\C$ is a flat right
$A$\+module (so the category of left $\C$\+comodules is abelian
by Lemma~\ref{comodule-category-lemma}(b)).
 Consider the following condition on the coring~$\C$:
\begin{itemize}
\item[($*$)] Every left $\C$\+comodule is a quotient comodule of
an $A$\+flat left $\C$\+comodule.
\end{itemize}
 Two classes of corings $\C$ known to satisfy~($*$) are described in
Remark~\ref{flat-sheaves-and-comodules} and
Lemma~\ref{finite-weak-dimension-lemma} (geometrically, these correspond
to quasi-compact semi-separated schemes and a certain kind of
smooth stacks).
 We are not aware of any example of a coring $\C$ over $A$ that is flat
as a (left and) right $A$\+module but does not satisfy~($*$).
{\hbadness=1500\par}

 We will also need a second condition:
\begin{itemize}
\item[(${*}{*}$)] Let $\M$ be a left $\C$\+comodule such that
the underlying $A$\+module of $\M$ has finite projective dimension
as an object of the abelian category $A\Modl$.
 Then $\M$ has a finite projective dimension as an object of
the abelian category $\C\Comodl$.
\end{itemize}

\begin{rem} \label{two-star-nontrivial}
 The corings $\C$ corresponding to quasi-compact semi-separated schemes
$X$ as per Example~\ref{kontsevich-rosenberg} satisfy~(${*}{*}$).
 This is the result of~\cite[Theorem~5.3]{PS6}.

 On the other hand, when $A=k$ is a field and $\C$ is a coalgebra
over~$k$, the condition~(${*}{*}$) asks the comodule category
$\C\Comodl$ to have finite homological dimension.
 This, of course, does \emph{not} hold in general.
 For example, if $\C$ is a finite-dimensional coalgebra over~$k$,
then the category of $\C$\+comodules is equivalent to the category
of modules over the $k$\+vector space dual finite-dimensional algebra,
$\C\Comodl\simeq\C^*\Modl$.
 Taking $\C^*$ to be a finite-dimensional $k$\+algebra of infinite
global dimension, one obtains an example when (${*}{*}$)~does not hold.

 Informally speaking, (${*}{*}$)~means that ``the coring~$\C$ has
finite homological dimension relative to the ring~$A$''.
 In the context of algebraic geometry, one can expect this to be
satisfied for many stacks in characteristic zero, such as
the quotient stacks $X/G$ of algebraic varieties $X$ by actions of
algebraic groups~$G$ (because the homological dimension of
the category of algebraic group representations does not exceed
the dimension of the group, in characteristic zero).
 On the other hand, in prime characteristic~$p$, the quotient of
a point by the action of a finite group of the order divisible by~$p$
does \emph{not} satisfy~(${*}{*}$), as the previous paragraph explains.

 See Remarks~\ref{two-star-cant-be-dropped-in-cotorsion-periodicity}
and~\ref{two-star-cant-be-dropped-in-flat-co-derived}
below for a further discussion.
\end{rem}

 We denote by $\Ext_\C^*({-},{-})$ the Ext groups in the abelian
category $\C\Comodl$.
 A left $\C$\+comodule $\B$ is said to be
\emph{cotorsion}~\cite[Section~C.2]{Pcosh} if $\Ext_\C^1(\F,\B)=0$
for all $A$\+flat left $\C$\+comodules~$\F$.
 Let $\sB=\C\Comodl^\cot$ denote the full subcategory of cotorsion
$\C$\+comodules in the abelian category $\sC=\C\Comodl$.
 Obviously, the full subcategory $\C\Comodl^\cot$ is closed
under extensions in $\C\Comodl$, so it inherits an exact category
structure from the abelian exact structure of $\C\Comodl$.

 We will say that a class of objects $\sF$ in an abelian category $\sC$
is \emph{resolving} if $\sF$ is closed under extensions and kernels of
epimorphisms in $\sC$, and every object of $\sC$ is a quotient object
of an object from~$\sF$.

 The following corollary is our version of~\cite[Corollary~5.6 and
Lemma~9.1]{PS6}.
 Part~(a) was already mentioned in the proof of
Theorem~\ref{flat-comodules-as-directed-colimits-theorem}.

\begin{cor} \label{prepare-for-comodule-cotorsion-periodicity}
 Let\/ $\C$ be a coring over a ring~$A$. \par
\textup{(a)} The class of all $A$\+flat left\/ $\C$\+comodules is
closed under directed colimits in the additive category\/ $\C\Comodl$. \par
\textup{(b)} Assume that the condition~\textup{($*$)} above on
the coring\/ $\C$ is satisfied.
 Then the class of all $A$\+flat left\/ $\C$\+comodules is resolving
in the abelian category\/ $\C\Comodl$.
 Consequently, one has\/ $\Ext_\C^n(\F,\B)=0$ for all $A$\+flat left\/
$\C$\+comodules\/ $\F$, all\/ cotorsion left\/ $\C$\+comodules\/ $\B$,
and all integers $n\ge1$. \par
\textup{(c)} Assume that the condition~\textup{(${*}{*}$)} above
on the coring\/ $\C$ is satisfied.
 Then any $A$\+flat left\/ $\C$\+comodule is an\/ $\aleph_1$\+directed
colimit of $A$\+flat\/ $\C$\+comodules having finite projective
dimensions as objects of\/ $\C\Comodl$.
\end{cor}

\begin{proof}
 Part~(a) follows from Lemma~\ref{comodule-category-lemma}(a), since
the class of all flat modules is closed under directed colimits
in $A\Modl$.
 In the context of part~(b), condition~($*$) presumes that $\C$ is
a flat right $A$\+module.
 Then the class of all $A$\+flat left $\C$\+comodules is closed under
extensions and kernels of epimorphisms in $\C\Comodl$ by
Lemma~\ref{comodule-category-lemma}(b), since the class of all flat
modules is closed under extensions and kernels of epimorphisms in
$A\Modl$.
 Every left $\C$\+comodule is an epimorphic image of an $A$\+flat
left $\C$\+comodule by the condition~($*$).
 The final assertion of~(b) is provided by~\cite[Lemma~6.1]{PS6}.

 In part~(c), all countably presentable flat $A$\+modules have
projective dimensions~$\le1$ in $A\Modl$ by
Lemma~\ref{count-pres-flat-mods}.
 Under condition~(${*}{*}$), it follows that all $A$\+countably
presentable $A$\+flat $\C$\+comodules have finite projective dimensions
in $\C\Comodl$.
 So it remains to refer to
Theorem~\ref{flat-comodules-as-directed-colimits-theorem}.
\end{proof}

 Asume that $\C$ is a flat right $A$\+module.
 Let $\sF=\C\Comodl_{A\dfl}$ denote the full subcategory of all
$A$\+flat $\C$\+comodules in the abelian category $\sC=\C\Comodl$.
 Since the full subcategory $\sF$ is closed under extensions in
$\C\Comodl$, it inherits an exact category structure from
the abelian exact structure of $\C\Comodl$.
 Since $\sF$ is also closed under directed colimits in $\sC$,
and the directed colimit functors are exact in $\sC=\C\Comodl$, it
follows that $\sF$ is an \emph{exact category with exact directed
colimits} in the sense of~\cite[Section~7]{PS6}, i.~e., all directed
colimits exist and directed colimits of admissible short exact
sequences are admissible short exact sequences in~$\sF$.

 Let $\sC$ be an abelian category and $\sB\subset\sC$ be a class of
objects.
 An object $M\in\sC$ is called \emph{$\sB$\+periodic} if there exists
a short exact sequence $0\rarrow M\rarrow B\rarrow M\rarrow0$ in
$\sC$ with $M$ as both the rightmost and the leftmost term,
and a middle object $B\in\sB$.

 The following theorem summarizes the technology of proving cotorsion
periodicity theorems going back to~\cite{BCE} and developed in its
present form in~\cite{PS6}.

\begin{thm} \label{general-cotorsion-periodicity-scheme}
 Let\/ $\sC$ be an abelian category with coproducts and\/ $\sF\subset
\sC$ be a resolving subcategory closed under directed colimits.
 Assume that the directed colimit functors in\/ $\sF$ take admissible
short exact sequences to admissible short exact sequences (in
the exact category structure on\/ $\sF$ inherited from the abelian
exact structure of\/ $\sC$, i.~e., the short sequences in\/ $\sF$
that are exact in\/~$\sC$).
 Assume further that the category\/ $\sF$ coincides with the directed
colimit closure of its full subcategory consisting of all the objects
of\/ $\sF$ that have finite projective dimensions in\/~$\sC$.
 Let\/ $\sB=\sF^{\perp_1}\subset\sC$ be the full subcategory of all
objects $B\in\sC$ such that\/ $\Ext^1_\sC(F,B)=0$ for all $F\in\sF$.
 Then any\/ $\sB$\+periodic object of\/ $\sC$ belongs to\/~$\sB$.
\end{thm}

\begin{proof}
 This is a particular case of~\cite[Corollary~8.2]{PS6}.
\end{proof}

 Now we can prove our cotorsion periodicity theorem for comodules.

\begin{thm} \label{comodule-cotorsion-periodicity-theorem}
 Let\/ $\C$ be a coring over an associative ring~$A$.
 Assume that both the conditions~\textup{($*$)}
and~\textup{(${*}{*}$)} on the coring\/ $\C$ are satisfied.
 Then any cotorsion-periodic left\/ $\C$\+comodule is cotorsion.
\end{thm}

\begin{proof}
 This is our generalization of~\cite[Theorem~9.2]{PS6}.
 The assumptions of Theorem~\ref{general-cotorsion-periodicity-scheme}
are satisfied by
Corollary~\ref{prepare-for-comodule-cotorsion-periodicity}
(cf.\ the discussion above).
\end{proof}

\begin{cor} \label{comodule-cotorsion-cocycles-cor}
 Let\/ $\C$ be a coring over an associative ring~$A$.
 Assume that both the conditions~\textup{($*$)}
and~\textup{(${*}{*}$)} on the coring\/ $\C$ are satisfied.
 Let\/ $\B^\bu$ be an acyclic complex in\/ $\C\Comodl$ whose terms\/
$\B^n$ are cotorsion left\/ $\C$\+comodules for all $n\in\boZ$.
 Then the\/ $\C$\+comodules of cocycles of the complex\/ $\B^\bu$ are
also cotorsion.
\end{cor}

\begin{proof}
 This is our generalization of~\cite[Corollary~9.4]{PS6}.
 To deduce it from Theorem~\ref{comodule-cotorsion-periodicity-theorem},
we apply~\cite[Proposition~8.4]{PS6}.
 It remains to check the assumptions of that proposition
from~\cite{PS6}.

 Indeed, the abelian category $\C\Comodl$ is Grothendieck by
Lemma~\ref{comodule-category-lemma}(b), so it has infinite products.
 The full subcategory $\sB$ of cotorsion left $\C$\+comodules
is closed under infinite products in the abelian category
$\sC=\C\Comodl$ by~\cite[Corollary 8.3]{CoFu}
or~\cite[Corollary A.2]{CoSt}.
 It is also obvious that $\sB$ is closed under direct summands in~$\sC$.
 
 Finally, any left $\C$\+comodule is a quotient comodule of
$A$\+flat $\C$\+comodule by the condition~($*$), and any $A$\+flat
$\C$\+comodule is a quotient comodule of a coproduct of
$A$\+countably presentable $A$\+flat left $\C$\+comodules by
Theorem~\ref{flat-comodules-as-directed-colimits-theorem}.
 The underlying $A$\+modules of the $\C$\+comodules of the latter kind
have projective dimension~$\le1$ by Lemma~\ref{count-pres-flat-mods}.
 Such $\C$\+comodules have finite projective dimensions in
$\C\Comodl$ by the condition~(${*}{*}$).
 Therefore, every left $\C$\+comodule is a quotient comodule of
an $A$\+flat left $\C$\+comodule having finite projective dimension
in $\C\Comodl$.
 It remains to recall that $\Ext_\C^n(\F,\B)=0$ for any $A$\+flat
left $\C$\+comodule $\F$, any left $\C$\+comodule $\B\in\sB$, and all
integers $n\ge1$ by
Corollary~\ref{prepare-for-comodule-cotorsion-periodicity}(b).
\end{proof}

\begin{cor} \label{comodule-cotorsion-derived-equivalence}
 Let\/ $\C$ be a coring over an associative ring~$A$.
 Assume that both the conditions~\textup{($*$)}
and~\textup{(${*}{*}$)} on the coring\/ $\C$ are satisfied.
 Then the inclusion of exact/abelian categories\/
$\C\Comodl^\cot\rarrow\C\Comodl$ induces an equivalence of their
unbounded derived categories,
$$
 \sD(\C\Comodl^\cot)\,\simeq\,\sD(\C\Comodl).
$$
\end{cor}

\begin{proof}
 This is our generalization of~\cite[Corollary~9.7]{PS6}.
 To deduce it from Corollary~\ref{comodule-cotorsion-cocycles-cor},
it suffices to refer to~\cite[Proposition~9.6]{PS6}.
\end{proof}

\begin{rem} \label{two-star-cant-be-dropped-in-cotorsion-periodicity}
 The assumption of condition~(${*}{*}$) \emph{cannot} be dropped in
any one of the assertions of
Theorem~\ref{comodule-cotorsion-periodicity-theorem},
Corollary~\ref{comodule-cotorsion-cocycles-cor},
or Corollary~\ref{comodule-cotorsion-derived-equivalence}.
 It suffices to choose a field $A=k$ and a finite-dimensional coalgebra
$\C$ over~$k$ whose dual algebra $\C^*$ is a Frobenius algebra of
infinite global dimension (e.~g., the algebra of dual numbers
$\C^*=k[\epsilon]/(\epsilon^2)$).
 Then $\C\Comodl=\C^*\Modl$ (cf.\ Remark~\ref{two-star-nontrivial}),
all $\C$\+comodules are $A$\+flat, and the cotorsion $\C$\+comodules
are the injective $\C$\+comodules, or equivalently the injective
$\C^*$\+modules.
 Pick any $\C^*$\+module of infinite projective/injective dimension;
then its two-sided projective-injective resolution is an unbounded
acyclic complex of injective $\C$\+comodules refuting the conclusions
of both Corollaries~\ref{comodule-cotorsion-cocycles-cor}
and~\ref{comodule-cotorsion-derived-equivalence}.
 Then it is clear from~\cite[proof of Proposition~7.6]{CH} or
\cite[Proposition~2]{EFI} that the conclusion of
Theorem~\ref{comodule-cotorsion-periodicity-theorem} does not hold
for the coalgebra/coring $\C$, either.
\end{rem}

\Section{Exact Sequences of $A$-Flat $\C$-Comodules~II}
\label{comodules-exact-sequences-II-secn}

 In this section we also deduce some results conditional upon
the assumptions ($*$) and~(${*}{*}$) from
Section~\ref{comodule-cotorsion-periodicity-secn} on a coring~$\C$.

 We start with an abstract category-theoretic discussion of absolutely
acyclic complexes in exact categories of finite homological dimension.
 Given an additive category $\sE$, we denote by $\sK(\sE)$
the triangulated homotopy category of (unbounded cochain
complexes in)~$\sE$.

 Let $\sE$ be an exact category.
 Any (termwise admissible) short exact sequence of complexes in $\sE$
can be viewed as a bicomplex with three rows, so one can construct its
total complex.
 A complex in $\sE$ is said to be \emph{absolutely
acyclic}~\cite[Sections~2 and~4]{Psemi}, \cite[Section~A.1]{Pcosh},
\cite[Section~5.1]{Pedg} if it belongs to the minimal thick subcategory
of the homotopy category $\sK(\sE)$ containing the totalizations of
short exact sequences in~$\sE$.
 In other words, a complex in $\sE$ is absolutely acyclic if and only if
it can be constructed as a homotopy direct summand of a complex obtained
from the totalizations of short exact sequences of complexes in $\sE$
by passing to cones of closed morphisms of complexes a finite number of times.

 Equivalently, the full subcategory of absolutely acyclic complexes
in $\Com(\sE)$ can be constructed as the closure of the full subcategory
of contractible complexes under extensions (in the termwise exact
structure on $\Com(\sE)$) and direct summands.
 This is essentially shown in~\cite[Proposition~8.12]{PS5}.

 Denote by $\Ext_\sE^*({-},{-})$ the Yoneda Ext functor in an exact
category~$\sE$.
 The exact category $\sE$ is said to have \emph{homological
dimension\/~$\le d$} (where $d\ge-1$ is an integer) if
$\Ext_\sE^n(X,Y)=0$ for all $X$, $Y\in\sE$ and $n>d$.

\begin{lem} \label{acyclic-are-absolutely-acyclic}
 In an exact category of finite homological dimension, all acyclic
complexes are absolutely acyclic.
\end{lem}

\begin{proof}
 This is~\cite[Lemma~2.1]{Psemi}.
\end{proof}

\begin{prop} \label{absolutely-acyclic-are-summands-of-totalizations}
 Let\/ $\sE$ be an exact category of homological dimension\/~$\le d$,
where $d$~is a finite integer.
 Then any absolutely acyclic complex in\/ $\sE$ is a direct summand of
the totalization of a $(d+\nobreak2)$\+term exact complex\/
$0\rarrow E^\bu_d\rarrow E^\bu_{d-1}\rarrow\dotsb\rarrow E^\bu_0
\rarrow E^\bu_{-1}\rarrow0$ of complexes in\/~$\sE$.
\end{prop}

\begin{proof}
 This result is a particular case of a similar assertion for
\emph{exact DG\+categories}, which is essentially proved
in~\cite[proof of Proposition~8.8]{Pedg}.
 More precisely, as a particular case of~\cite[proof of
Proposition~8.8]{Pedg} (corresponding to~\cite[Examples~4.39
and~6.1(4)]{Pedg}) one obtains the claim that any absolutely
acyclic complex $A^\bu$ in $\sE$ is a \emph{homotopy direct summand} of
the totalization $T^\bu$ of a $(d+\nobreak2)$\+term exact complex of
complexes in~$\sE$.
 Here a ``homotopy direct summand'' means a direct summand in
the homotopy category~$\sK(\sE)$.
 It remains to show that a homotopy direct summand can be turned into
a strict direct summand.

 Now we have a pair of closed morphisms of complexes $i\:A^\bu
\rarrow T^\bu$ and $p\:T^\bu\rarrow A^\bu$ such that the composition
$pi\:A^\bu\rarrow A^\bu$ is homotopic to the identity morphism
$\id\:A^\bu\rarrow A^\bu$.
 It follows that the difference $pi-\id\:A^\bu\rarrow A^\bu$ factorizes
as $A^\bu\rarrow C^\bu\rarrow A^\bu$, where $C^\bu$ is the cone of
identity morphism $\id\:A^\bu\rarrow A^\bu$ and $A^\bu\rarrow C^\bu$
is the natural closed morphism, while $C^\bu\rarrow A^\bu$ is some
arbitrary closed morphism.
 Hence the identity morphism $\id\:A^\bu\rarrow A^\bu$ factorizes
through the direct sum $T^\bu\oplus C^\bu$ as a morphism
in the category of complexes in~$\sE$.
 Finally, we point out that $C^\bu$ is the totalization of
a $2$\+term exact complex $0\rarrow A^\bu\rarrow A^\bu\rarrow0$
of complexes in~$\sE$.
 This takes care of all cases with $d\ge0$, while in the trivial
case of $d=-1$ we have $\sE=0$ and there is nothing to prove.
\end{proof}

 Let $\sF$ be an exact category.
 A class of objects $\sG\subset\sF$ is said to be
\emph{self-generating}~\cite[Section~6]{PS6} if for any admissible
epimorphism $F\rarrow H$ in $\sF$ with $H\in\sG$ there exists a morphism
$G\rarrow F$ in $\sF$ with $G\in\sG$ such that the composition
$G\rarrow F\rarrow H$ is an admissible epimorphism in~$\sF$.
 A class of objects $\sG$ in $\sF$ is said to be
\emph{self-resolving}~\cite[Section~7.1]{Pedg}, \cite[Section~8]{PS6}
if it is self-generating, closed under extensions, and closed under
kernels of admissible epimorphisms.
 
 Clearly, any self-resolving full subcategory $\sG$ in an exact
category $\sF$ interits an exact category structure from $\sF$ (since
$\sG$ is closed under extensions in~$\sF$).
 The following lemma is straightforward and well-known.

\begin{lem} \label{self-resolving-ext-isomorphism}
 Let\/ $\sF$ be an exact category and\/ $\sG\subset\sF$ be
a self-resolving full subcategory.
 Then the inclusion of exact categories\/ $\sG\rarrow\sF$ induces
isomorphisms on the\/ $\Ext$ groups,
$$
 \Ext_\sG^n(X,Y)\simeq\Ext_\sF^n(X,Y)
 \qquad\text{for all\/ $X$, $Y\in\sG$ and\/ $n\ge0$}.
$$
\end{lem}

\begin{proof}
 See~\cite[Section~12]{Kel} or~\cite[Proposition~A.2.1]{Pcosh}.
\end{proof}

 Now we return to comodules over a coring $\C$ over an associative
ring~$A$.

\begin{cor} \label{comodules-count-pres-flat-exact-category}
 Let\/ $\C$ be a coring over an associative ring $A$ such that\/ $\C$
is a flat right $A$\+module. \par
\textup{(a)} The full subcategory\/ $\sG$ of all $A$\+countably
presentable $A$\+flat left\/ $\C$\+comodules is self-resolving in
the exact category\/ $\sF$ of all $A$\+flat left\/ $\C$\+comodules. \par
\textup{(b)} If the conditions \textup{($*$)} and~\textup{(${*}{*}$)}
on the coring\/ $\C$ are satisfied, then the exact category\/ $\sG$
of $A$\+countably presentable $A$\+flat left\/ $\C$\+comodules has
finite homological dimension.
\end{cor}

\begin{proof}
 In part~(a), the assertion that $\sG$ is self-generating in $\sF$ is
an easy corollary of
Theorem~\ref{flat-comodules-as-directed-colimits-theorem}.
 Let $\HH$ be an $A$\+countably presentable $A$\+flat left
$\C$\+comodule, $\F$ be an $A$\+flat $\C$\+comodule, and $\F\rarrow\HH$
be a surjective $\C$\+comodule morphism.
 By Theorem~\ref{flat-comodules-as-directed-colimits-theorem},
there exists an $\aleph_1$\+directed poset $\Xi$ and a $\Xi$\+indexed
diagram of $A$\+countably presentable $A$\+flat $\C$\+comodules
$(\G_\xi)_{\xi\in\Xi}$ for which $\F=\varinjlim_{\xi\in\Xi}\G_\xi$.
 Then it is clear that there exists an index $\xi_0\in\Xi$ such that
the composition $\G_{\xi_0}\rarrow\F\rarrow\HH$ is surjective.
 The full subcategory $\sG$ is closed under kernels of surjective
homomorphisms in $\sF$ by Lemma~\ref{countable-flat-coherence},
and the assertion that $\sG$ is closed under extensions in $\sF$ is
provable similarly to
Corollary~\ref{prepare-for-comodule-cotorsion-periodicity}(b).

 In part~(b), the full subcategory $\sG$ is self-resolving in $\sF$ by
part~(a), and the full subcategory $\sF$ is resolving in the abelian
category $\sC=\C\Comodl$
by Corollary~\ref{prepare-for-comodule-cotorsion-periodicity}(b).
 Consequently, $\sG$ is self-resolving in~$\sC$.
 The underlying $A$\+modules of all the $\C$\+comodules from $\sG$ have
projective dimensions at most~$1$ by
Lemma~\ref{count-pres-flat-mods}, hence the objects of $\sG$ has
finite projective dimensions in $\sC$ by condition~(${*}{*}$).
 Moreover, the latter projective dimensions must be uniformly bounded
by some fixed integer~$d$, since the class of objects $\sG$ is
closed under countable coproducts in~$\sC$.
 Finally, Lemma~\ref{self-resolving-ext-isomorphism} tells us that
the Ext groups computed in $\sG$ and in $\sC$ agree.
 Thus the homological dimension of the exact category $\sG$ cannot
exceed~$d$.
\end{proof}

 The following theorem is a comodule version
of~\cite[Theorem~2.4\,(1)\,$\Leftrightarrow$\,(3)]{EG} or
\cite[Theorem~8.6\,(ii)\,$\Leftrightarrow$\,(iii)]{Neem}.
 It should be also compared to~\cite[Corollary~0.5 or
Proposition~8.13]{PS5}.
 In the latter context, one can say that
Theorem~\ref{flat-comodule-pure-acyclic-complexes-characterized} means
that (under the conditions ($*$) and~(${*}{*}$))
\emph{for the exact category of $A$\+flat\/ $\C$\+comodules,
the derived category coincides with the} (\emph{suitably defined})
\emph{coderived category}.

\begin{thm} \label{flat-comodule-pure-acyclic-complexes-characterized}
 Let\/ $\C$ be a coring over an associative ring~$A$.
 Assume that both the conditions \textup{($*$)} and~\textup{(${*}{*}$)}
from Section~\ref{comodule-cotorsion-periodicity-secn}
on the coring\/ $\C$ are satisfied.
 Then there exists a finite integer $d\ge0$ such that the following
classes of complexes of left\/ $\C$\+comodules coincide:
\begin{enumerate}
\item all $A$\+pure acyclic complexes of $A$\+flat\/ $\C$\+comodules;
\item the closure of the class of all contractible complexes of
$A$\+countably presentable $A$\+flat\/ $\C$\+comodules under extensions
and directed colimits;
\item all\/ $\aleph_1$\+directed colimits of totalizations of
$(d+\nobreak2)$\+term exact complexes of complexes of $A$\+countably
presentable $A$\+flat\/ $\C$\+comodules.
\end{enumerate}
\end{thm}

\begin{proof}
 One can easily see that the totalization of a $(d+\nobreak2)$\+term
exact complex of complexes is a finitely iterated extension of $d+1$
cones of identity endomorphisms of complexes.
 Furthermore, $\aleph_1$\+directed colimits are a subclass of directed
colimits, the class of all $A$\+pure acyclic complexes of $A$\+flat
$\C$\+comodules is closed under extensions and directed colimits, and
the contractible complexes of $A$\+flat $\C$\+comodules are $A$\+pure
acyclic.
 This proves the inclusions (3)\,$\subset$\,(2)\,$\subset$\,(1).

 Conversely, to prove that (1)\,$\subset$\,(3), recall that all
$A$\+pure acyclic complexes of $A$\+flat $\C$\+comodules are
$\aleph_1$\+directed colimits of $A$\+pure acyclic complexes of
$A$\+countably presentable $A$\+flat $\C$\+comodules by
Corollary~\ref{comods-pure-acycl-complexes-as-aleph1-dir-colims}.
 The exact category $\sG$ of $A$\+countably presentable $A$\+flat
$\C$\+comodules has finite homological dimension by
Corollary~\ref{comodules-count-pres-flat-exact-category}(b).
 Hence the $A$\+pure acyclic complexes of $A$\+countably presentable
$A$\+flat $\C$\+comodules are absolutely acyclic
by Lemma~\ref{acyclic-are-absolutely-acyclic}, and it remains to refer
to Proposition~\ref{absolutely-acyclic-are-summands-of-totalizations}.
\end{proof}

\begin{rem} \label{two-star-cant-be-dropped-in-flat-co-derived}
 Without the assumption~(${*}{*}$), the conclusion of
Theorem~\ref{flat-comodule-pure-acyclic-complexes-characterized}
certainly does \emph{not} hold (cf.\ Remarks~\ref{two-star-nontrivial}
and~\ref{two-star-cant-be-dropped-in-cotorsion-periodicity}).
 In fact, for a coalgebra $\C$ over a field $A=k$, if the global
dimension of $\C$ is infinite, then the coderived category of
$\C$\+comodules is usually quite different from their derived
category~\cite[Section~7]{Pksurv}.
\end{rem}

\Section{Preliminaries on Topological Rings and Contramodules}
\label{contra-preliminaries-secn}

 The material of this section goes back to~\cite[Remark~A.3]{Psemi},
\cite[Appendix~E]{Pcosh}, \cite[Section~2.1]{Prev},
and~\cite[Section~5]{PR}.
 For more recent references, see~\cite[Section~2]{Pproperf}
or~\cite[Section~2]{Pcoun}.
 The exposition in~\cite[Sections~6\+-7]{PS1} may be the most
accessible one for a beginner.
 In this paper, we are mostly interested in contramodules over
complete, separated topological associative rings with a countable
base of neighborhoods of zero consisting of open two-sided ideals.

 Let $\fA$ be a topological abelian group where open subgroups form
a base of neighborhoods of zero.
 Consider the natural map to the directed limit $\lambda_\fA\:\fA
\rarrow\varprojlim_{\fU\subset\fA}\fA/\fU$, where $\fU$ ranges over
the open subgroups of~$\fA$.
 The topological group $\fA$ is said to be \emph{separated} if
the map~$\lambda_\fA$ is injective, and \emph{complete} if
$\lambda_\fA$~is surjective.

 Let $\R$ be a complete, separated topological ring with a base of
neighborhoods of zero consisting of open right ideals.
 A right $\R$\+module $\N$ is said to be \emph{discrete} if the action
map $\N\times\R\rarrow\N$ is continuous in the given topology of $\R$
and the discrete topology on~$\N$.
 Equivalently, a right $\R$\+module $\N$ is discrete if, for every
element $x\in\N$, the annihilator of~$x$ in $\R$ is an open right ideal.
 The full subcategory of discrete right $\R$\+modules $\Discr\R\subset
\Modr\R$ is a Grothendieck abelian category.
 The inclusion functor $\Discr\R\rarrow\Modr\R$ is exact and preserves
all colimits.

 Given an abelian group $A$ and a set $X$, we use $A[X]=A^{(X)}$ as
a notation for the direct sum of $X$ copies of~$A$.
 The elements of $A[X]$ are interpreted as finitely supported formal
linear combinations $\sum_{x\in X}a_xx$ of elements of $X$ with
the coefficients $a_x\in A$ (where $a_x=0$ for all but a finite subset
of indices $x\in X$).
 
 Let $\fA$ be a complete, separated topological abelian group with a base
of neighborhoods of zero consisting of open subgroups.
 Given a set $X$, we denote by $\fA[[X]]$ the directed limit of abelian
groups $\varprojlim_{\fU\subset\fA}(\fA/\fU)[X]$ (where $\fU$ ranges
over the open subgroups of~$\fA$).
 The elements of $\fA[[X]]$ are interpreted as infinite formal linear
combinations $\sum_{x\in X}a_xx$ with the families of coefficients
$a_x\in\fA$ \emph{converging to zero in the topology of\/~$\fA$}.
 The latter condition means that, for every open subgroup
$\fU\subset\fA$, one has $a_x\in\fU$ for all but a finite subset of
indices $x\in X$.
 
 Given a map of sets $X\rarrow Y$, the induced (pushforward) map
$\fA[[f]]\:\fA[[X]]\rarrow\fA[[Y]]$ is defined by the rule
$\sum_{x\in X}a_xx\longmapsto\sum_{y\in Y}b_yy$, where
$b_y=\sum_{x\in X}^{f(x)=y}a_x$ for every $y\in Y$.
 Here the latter infinite sum (defining the element $b_y\in\fA$) is
\emph{not} formal; rather, it is understood as the limit of finite
partial sums in the topology of~$\fA$.
 Due to this construction, the assigment $X\longmapsto\fA[[X]]$ is
a covariant functor $\Sets\rarrow\Ab$ from the category of sets to
the category of abelian groups.
 We will mostly view it as an endofunctor on the category of sets,
$\fA[[{-}]]\:\Sets\rarrow\Sets$, ignoring the abelian group structures
on the sets $\fA[[X]]$.

 Let $\R$ be a complete, separated topological ring where open right
ideals form a base of neighborhoods of zero.
 Then the functor $\R[[{-}]]\:\Sets\rarrow\Sets$ acquires a natural
structure of a \emph{monad} on the category of sets.
 We refer to~\cite[Chapter~VI]{MacLane}
(see also~\cite[Section~6]{PS1}) for the background discussion of
monads and algebras over monads. 
 Specifically, for any set $X$, the monad unit $\epsilon_X\:X
\rarrow\R[[X]]$ is the obvious natural map taking every element $x\in X$
to the formal linear combination $\sum_{y\in X}r_yy$, where $r_x=1$
and $r_y=0$ for all $y\ne x$.

 The monad multiplication $\phi_X\:\R[[\R[[X]]]]\rarrow\R[[X]]$ is
the natural ``opening of parentheses'' map producing a formal linear
combination from a formal linear combination of formal linear
combinations.
 This construction involves taking products of pairs of elements in $\R$
and infinite sums of zero-convergent families of elements.
 The conditions that $\R$ is complete, separated, and open right ideals
form a base of neighborhoods of zero guarantee the convergence.

 \emph{Left\/ $\R$\+contramodules} are defined as algebras over
the monad $(\R[[{-}]],\epsilon,\phi)$ on the category of sets.
 We usually prefer to speak of \emph{modules} (rather than algebras)
\emph{over monads} in this context, as contramodules are more similar
to modules over rings than to algebras.
 So, specifically, a left $\R$\+contramodule $\fP$ is a set endowed
with a \emph{left contraaction} map $\pi\:\R[[\fP]]\rarrow\fP$
satisfying the following \emph{contraassociativity} and
\emph{contraunitality} axioms.
 The two compositions $\pi\circ\phi_\fP$ and $\pi\circ\R[[\pi]]$ must
be equal to each other,
$$
 \R[[\R[[\fP]]]]\rightrightarrows\R[[\fP]]\rarrow\fP,
$$
and the composition $\pi\circ\epsilon_\fP$ must be ideal to
the identity map~$\id_\fP$,
$$
 \fP\rarrow\R[[\fP]]\rarrow\fP.
$$
 We denote the category of left $\R$\+contramodules by $\R\Contra$.

 Restricting the contraaction map $\pi\:\R[[\fP]]\rarrow\fP$ to
the subset of finitely supported formal linear combinations
$\R[\fP]\subset\R[[\fP]]$, one constructs the underlying left
$\R$\+module structure on a left $\R$\+contramodule~$\fP$
\,\cite[Sections~6.1\+-6.2]{PS1}.
 Hence the forgetful functor $\R\Contra\rarrow\R\Modl$.
 The category of left $\R$\+contramodules is abelian.
 All set-indexed products and coproducts (hence also limits and
colimits) exist in $\R\Contra$.
 The forgetful functor $\R\Contra\rarrow\R\Modl$ is exact, faithful,
and preserves infinite products.
 We use the notation $\Hom^\R({-},{-})$ for the groups of morphisms
in the abelian category $\R\Contra$.

 For any set $X$, the set/abelian group $\R[[X]]$ has a natural
left $\R$\+contramodule structure with the contraaction map
$\pi=\phi_X\:\R[[\R[[X]]]]\rarrow\R[[X]]$.
 The $\R$\+con\-tra\-mod\-ules $\R[[X]]$ are called the \emph{free}
left $\R$\+contramodules.
 For any left $\R$\+con\-tra\-mod\-ule $\fP$, the $\R$\+contramodule
morphisms $\R[[X]]\rarrow\fP$ correspond bijectively to maps of sets
$X\rarrow\fP$; so there is a natural isomorphism of abelian groups
$$
 \Hom^\R(\R[[X]],\fP)\simeq\Hom_\Sets(X,\fP).
$$
 There are enough projective objects in the abelian category
$\R\Contra$; the projective $\R$\+contramodules are precisely
the direct summands of the free ones.

 Given a left $\R$\+contramodule $\fP$ and a closed subgroup
$\fA\subset\R$, we denote by $\fA\tim\fP\subset\fP$ the image of
the composition $\fA[[\fP]]\rarrow\R[[\fP]]\rarrow\fP$ of the natural
inclusion $\fA[[\fP]]\hookrightarrow\R[[\fP]]$ with the contraaction
map $\pi\:\R[[\fP]]\rarrow\fP$.
 So $\fA\tim\fP$ is a subgroup in~$\fP$ (since $\pi$~is a left
$\R$\+contramodule morphism, hence in particular an abelian group
homomorphism, and $\fA[[\fP]]$ is subgroup in $\R[[\fP]]$).
 If $\fJ\subset\R$ is a closed left ideal, then $\fJ\tim\fP$ is
an $\R$\+subcontramodule in~$\fP$ (since $\fJ[[\fP]]$ is
an $\R$\+subcontramodule in the free contramodule $\R[[\fP]]$).

 Let $\N$ be a discrete right $\R$\+module and $\fP$ be a left
$\R$\+contramodule.
 Then the \emph{contratensor product} $\N\ocn_\R\fP$ is the abelian
group constructed as the cokernel of (the difference of)
the natural pair of maps
$$
 \N\ot_\boZ\R[[\fP]]\,\rightrightarrows\,\N\ot_\boZ\fP.
$$
 Here the first map $\N\ot_\boZ\R[[\fP]]\rarrow\N\ot_\boZ\fP$ is
induced by the contraaction map $\pi\:\R[[\fP]]\rarrow\fP$, while
the second map $\N\ot_\boZ\R[[\fP]]\rarrow\N\ot_\boZ\fP$ is
defined by the rule
$$
 y\ot\sum\nolimits_{p\in\fP}r_pp\longmapsto
 \sum\nolimits_{p\in\fP}yr_p\ot p
$$
for all $y\in\N$ and $\sum_{p\in\fP}r_pp\in\R[[\fP]]$.
 Here the sum in the right-hand side is well-defined, because one has
$yr_p=0$ in $\N$ for all but a finite subset of indices $p\in\fP$.

 The contratensor product functor $\ocn_\R\:\Discr\R\times\R\Contra
\rarrow\Ab$ preserves all colimits in both the arguments.
 For any discrete right $\R$\+module $\N$ and any set $X$, one has
a natural isomorphism of abelian groups
$$
 \N\ocn_\R\R[[X]]\simeq\N[X]=\N^{(X)}.
$$
 For any closed right ideal $\fJ\subset\R$, one has
$$
 \fJ\tim(\R[[X]])=\fJ[[X]]\subset\R[[X]].
$$
 For any open right ideal $\fI\subset\R$ and any left
$\R$\+contamodule $\fP$, one has
$$
 (\R/\fI)\ocn_\R\fP\simeq\fP/(\fI\tim\fP).
$$
 In particular, $(\R/\fI)\ocn_\R\R[[X]]\simeq\R[[X]]/(\fI\tim\R[[X]])
=\R[[X]]/\fI[[X]]\simeq(\R/\fI)[X]$.

 For any open two-sided ideal $\fI\subset\R$ and any left
$\R$\+contramodule $\fP$, the quotient group $\fP/(\fI\tim\fP)$
has a natural module structure over the quotient ring~$\R/\fI$.
 The subgroup $\fI\tim\fP$ is an $\R$\+subcontramodule in $\fP$,
as mentioned above; and the quotient $\R$\+contramodule structure
on $\fP/(\fI\tim\fP)$ comes from the $\R/\fI$\+module structure.

\Section{Separated and Flat Contramodules}
\label{separated-and-flat-contra-secn}

 In this section we discuss contramodules over a topological ring
with a \emph{countable} base of neighborhoods of zero.
 There are no new results in this section, which is based
on~\cite[Section~E.1]{Pcosh} and~\cite[Section~6]{PR}.
 A brief survey is available in~\cite[Section~3.3]{Prev}.
 A more recent exposition can be found in~\cite[Sections~4\+-5]{Pcoun}.
 A connection with algebraic geometry is discussed
in~\cite[Chapter~3]{Psemten}.

 Given a left $\R$\+contramodule $\fP$, consider the natural map to
the directed limit $\lambda_{\R,\fP}\:\fP\rarrow
\varprojlim_{\fI\subset\R}\fP/(\fI\tim\fP)$, where $\fI$ ranges
over the open right ideals of~$\R$.
 A left $\R$\+contramodule $\fP$ is said to be \emph{separated} if
the map~$\lambda_{\R,\fP}$ is injective, and \emph{complete} if
$\lambda_{\R,\fP}$ is surjective.
 We denote full subcategory of separated $\R$\+contamodules by
$\R\Contra_\sep\subset\R\Contra$.

\begin{lem} \label{all-contramodules-complete}
 Let\/ $\R$ be a complete, separated topological ring with
a \emph{countable} base of neighborhoods of zero consisting of open
right ideals.
 Then all left\/ $\R$\+contramodules are complete (but they
\emph{need not} be separated).
\end{lem}

\begin{proof}
 This goes back, at least, to~\cite[Lemma~A.2.3 and Remark~A.3]{Psemi}.
 Counterexamples of nonseparated contramodules appeared
in~\cite[Section~A.1]{Psemi} (see also~\cite[Example~2.5]{Sim},
\cite[Example~3.20]{Yek}, \cite[Section~1.5]{Prev},
\cite[Example~2.7(1)]{Pcta}, and a further discussion in
Remark~\ref{separated-contramodules-not-well-behaved} below).
 The completeness assertion for contramodules over topological rings
with a countable topology base consisting of two-sided ideals can be
found in~\cite[Lemma~E.1.1]{Pcosh}, and the general case of a countable
topology base of right ideals is covered by~\cite[Lemma~6.3(b)]{PR}.
 See also~\cite[Lemmas~3.6 and~3.9]{Prev}.
\end{proof}

\begin{lem} \label{contramodule-nakayama}
 Let\/ $\R$ be a complete, separated topological ring with
a \emph{countable} base of neighborhoods of zero consisting of open
right ideals.
 Then, for any nonzero left\/ $\R$\+contramodule\/ $\fP$, there exists
an open right ideal\/ $\fI\subset\R$ such that\/ $\fP/(\fI\tim\fP)\ne0$.
\end{lem}

\begin{proof}
 This assertion belongs to a class of results known as ``contramodule
Nakayama lemmas'' and going back, at least,
to~\cite[Lemma~A.2.1]{Psemi}.
 The assertion for contramodules over topological rings with
a countable topology base of two-sided ideals can be found
in~\cite[Lemma~E.1.2]{Pcosh}, and the general case of a countable base
of right ideals is covered by~\cite[Lemma~6.14]{PR}.
 See also~\cite[Lemmas~2.1(b) and~3.22]{Prev}.
\end{proof}

 A left $\R$\+contramodule $\fF$ is called \emph{flat} if
the contratensor product functor ${-}\ocn_\R\nobreak\fF\:
\allowbreak\Discr\R\rarrow\Ab$ is exact on the abelian category of
discrete right $\R$\+modules.
 All projective $\R$\+contramodules are separated and flat.
 All directed colimits and coproducts of flat $\R$\+contramodules
(computed in the category $\R\Contra$) are flat.
 We denote the full subcategory of flat $\R$\+contramodules by
$\R\Contra_\flat\subset\R\Contra$.

 When $\R$ has a base of neighborhoods of zero consisting of
open two-sided ideals~$\fI$, the discrete right $\R$\+modules are
the directed unions of modules over the quotient rings~$\R/\fI$.
 For any right $\R/\fI$\+module $N$ and any left $\R$\+contramodule
$\fP$, one has
$$
 N\ocn_\R\fP\simeq N\ot_{\R/\fI}\fP/(\fI\tim\fP).
$$
 Hence, in this case, a left $\R$\+contramodule $\fF$ is flat
if and only if the left $\R/\fI$\+module $\fF/(\fI\tim\fF)$ is flat
for every open two-sided ideal $\fI\subset\R$.

\begin{lem} \label{flat-contramodules-separated}
 Let\/ $\R$ be a complete, separated topological ring with
a \emph{countable} base of neighborhoods of zero consisting of open
right ideals.
 Then all \emph{flat} left\/ $\R$\+con\-tra\-mod\-ules are separated.
\end{lem}

\begin{proof}
 The case of topological rings with a countable topology base of
two-sided ideals can be found in~\cite[Lemma~E.1.7]{Pcosh}
(cf.\ the definition of a ``flat contramodule''
in~\cite[the paragraph before Lemma~E.1.4]{Pcosh}, which is stated
differently than our definition above).
 The general case of a countable base of right ideals is covered
by~\cite[Corollary~6.15]{PR}.
 See also~\cite[Propositions~3.7(a) and~3.10(a)]{Prev}.
\end{proof}

\begin{lem} \label{flat-contramodules-well-behaved}
 Let\/ $\R$ be a complete, separated topological ring with
a \emph{countable} base of neighborhoods of zero consisting of open
right ideals.
 Then \par
\textup{(a)} the class of all flat left\/ $\R$\+contramodules is closed
under extensions and kernels of epimorphisms in\/ $\R\Contra$; \par
\textup{(b)} for any discrete right\/ $\R$\+module\/ $\N$,
the contratensor product functor\/ $\N\ocn_\R\nobreak{-}\,\:\allowbreak
\R\Contra\rarrow\Ab$ takes short exact sequences of flat left\/
$\R$\+contramodules to short exact sequences of abelian groups.
\end{lem}

\begin{proof}
 The case of topological rings with a countable topology base of
two-sided ideals can be found in~\cite[Lemmas~E.1.4 and~E.1.5]{Pcosh}.
 The general case of a countable base of right ideals is covered
by~\cite[Lemma~6.7 or~6.10, Corollaries~6.8, 6.13 and~6.15]{PR}, with
a summary in~\cite[Corollary~7.1(a\+-b)]{PR}.
 See also~\cite[Propositions~3.7(b) and~3.10(b)]{Prev}.
\end{proof}

 In particular, according to
Lemma~\ref{flat-contramodules-well-behaved}(a), the full subcategory of
flat $\R$\+contramod\-ules $\R\Contra_\flat$ is closed under extensions
in $\R\Contra$ when the topological ring $\R$ has a countable base of
neighborhoods of zero.
 So the full subcategory\/ $\R\Contra_\flat$ inherits an exact category
structure from the abelian exact structure of $\R\Contra$ in this case.

\begin{rem} \label{separated-contramodules-not-well-behaved}
 Lemma~\ref{flat-contramodules-well-behaved} tells us that the class
of flat $\R$\+contramodules is well-behaved (under the assumption of
a countable topology base in~$\R$).
 But the class of separated $\R$\+contramodules is \emph{not}
well-behaved.

 In suffices to consider the case of the ring of $p$\+adic integers
$\R=\boZ_p$ (in the $p$\+adic topology), or the ring of formal Taylor
power series $\R=k[[t]]$ in one variable~$t$ over a field~$k$ (in
the $t$\+adic topology).
 In both cases, the forgetful functor $\R\Contra\rarrow\R\Modl$ is
fully faithful; in fact, even the forgetful functors $\boZ_p\Contra
\rarrow\Ab$ and $k[[t]]\Contra\rarrow k[t]\Modl$ are fully faithful.
 The essential images of these forgetful functors are described
in~\cite[Section~2.2]{Prev} and the references therein.

 With these descriptions in mind, there is a now-classical
counterexample showing that the the full subcategory $\R\Contra_\sep$
is not closed under extensions in $\R\Contra$ \,\cite[Example~2.5]{Sim}
and the category $\R\Contra_\sep$ is \emph{not}
abelian~\cite[Example~2.7(1)]{Pcta}.
 Furthermore, the full subcategory $\R\Contra_\sep$ is \emph{not}
closed under directed colimits, and in fact, \emph{not} even closed
under coproducts in $\R\Contra$ \,\cite[Section~1.5]{Prev}.
 (See also~\cite[Example~3.1]{Psemten} and
Example~\ref{ind-affine-ind-scheme-example} below.)

 However, separated contramodules are sometimes easier to work with
than arbitrary ones, as the following lemma illustrates.
\end{rem}

\begin{lem} \label{separated-and-flat-contramodules-by-reduction-data}
 Let\/ $\R$ be a complete, separated topological ring, and let\/
$\R\supset\fI_1\supset\fI_2\supset\fI_3\supset\dotsb$ be a countable
base of neighborhoods of zero in\/ $\R$ consisting of open
\emph{two-sided} ideals\/ $\fI_n\subset\R$, \,$n\ge1$. \par
\textup{(a)} The full subcategory\/ $\R\Contra_\sep$ of separated\/
$\R$\+contramodules\/ $\fM$ in\/ $\R\Contra$ is equivalent to
the category formed by the following sets of data:
\begin{enumerate}
\item for every $n\ge1$, a left\/ $\R/\fI_n$\+module $M_n$ is given; and
\item for every $n\ge1$, an isomorphism of left\/ $\R/\fI_n$\+modules
$$
 M_n\,\simeq\,\R/\fI_n\ot_{\R/\fI_{n+1}}M_{n+1}
$$
is given.
\end{enumerate} \par
\textup{(b)} The full subcategory\/ $\R\Contra_\flat$ of flat\/
$\R$\+contramodules\/ $\fF$ in\/ $\R\Contra$ is equivalent to
the category formed by the following sets of data:
\begin{enumerate}
\item for every $n\ge1$, a flat left\/ $\R/\fI_n$\+module $F_n$
is given; and
\item for every $n\ge1$, an isomorphism of left\/ $\R/\fI_n$\+modules
$$
 F_n\,\simeq\,\R/\fI_n\ot_{\R/\fI_{n+1}}F_{n+1}
$$
is given.
\end{enumerate} \par
 Here, in both parts~(a) and~(b), morphisms between the described sets
of data are defined in the obvious way as collections module morphisms
forming commutative diagrams with the reduction isomorphisms.
\end{lem}

\begin{proof}
 In part~(a), the equivalence of categories assigns to a separated
left $\R$\+con\-tra\-mod\-ule $\fM$ the collection of modules
$M_n=(\R/\fI_n)\ocn_\R\fM=\fM/(\fI_n\tim\fM)$, endowed with the obvious
reduction isomorphisms.
 The inverse functor assigns to a collection of modules $M_n$ endowed
with reduction isomorphisms the separated $\R$\+contramodule
$\fM=\varprojlim_{n\ge1}M_n$.
 Here the $\R/\fI_n$\+modules $M_n$ are viewed as $\R$\+contramodules
via the natural surjective ring homomorphisms $\R\rarrow\R/\fI_n$, and
the directed limit can be computed in the category $\R\Contra$ (as
the forgetful functor $\R\Contra\rarrow\R\Modl$ preserves all limits).
 The assertion that these rules indeed define mutually inverse
equivalences of categories is~\cite[Lemma~E.1.3]{Pcosh}
(see~\cite[Corollary~6.4]{PR} for a generalization to topological rings
with a countable base of right ideals).
 Lemma~\ref{all-contramodules-complete} needs to be used here.

 The equivalence of categories in part~(b) is obtained by restricting
the equivalence of part~(a) to the respective full subcategories on
both sides.
 Lemma~\ref{flat-contramodules-separated} plays an important role here.
\end{proof}

\begin{lem} \label{projective-contramodules-by-reduction-data}
 Let\/ $\R$ be a complete, separated topological ring, and let\/
$\R\supset\fI_1\supset\fI_2\supset\fI_3\supset\dotsb$ be a countable
base of neighborhoods of zero in\/ $\R$ consisting of open
\emph{two-sided} ideals\/ $\fI_n\subset\R$, \,$n\ge1$.
 Then a left\/ $\R$\+contramodule\/ $\fP$ is projective if and only if
(it is separated and flat and) the corresponding reduction datum in
Lemma~\ref{separated-and-flat-contramodules-by-reduction-data} consists
of projective modules, that is, if and only if
the\/ $\R/\fI_n$\+module\/ $\fP/(\fI_n\tim\fP)$ is projective
for every $n\ge1$.
\end{lem}

\begin{proof}
 This is~\cite[Corollary~E.1.10(a)]{Pcosh};
see also~\cite[Proposition~3.8]{Prev}.
\end{proof}

 The category $\R\Contra$ is better behaved than $\R\Contra_\sep$ in
that $\R\Contra$ is abelian (cf.\
Remark~\ref{separated-contramodules-not-well-behaved}).
 Still, the directed colimit functors are \emph{not} exact
in $\R\Contra$, as one can see, e.~g., from~\cite[Examples~4.4]{PR}.
 The next lemma tells us that the directed colimit functors \emph{are}
exact in the exact category $\R\Contra_\flat$ (assuming
a countable base of neighborhoods of zero in~$\R$).

\begin{lem} \label{directed-colimits-of-flat-contramodules-exact}
 Let\/ $\R$ be a complete, separated topological ring with
a \emph{countable} base of neighborhoods of zero consisting of open
right ideals.
 Then the directed colimit of any directed diagram of short exact
sequences of \emph{flat} left\/ $\R$\+contramodules (computed in
the category\/ $\R\Contra$) is a short exact sequence of (flat)
left\/ $\R$\+contramodules.
\end{lem}

\begin{proof}
 In the case of a topological ring with a countable topology base of
two-sided ideals, the assertion is easily deduced
from the fact that the class of all flat left $\R$\+contramodules
is closed under directed colimits in $\R\Contra$ and the description of
flat left $\R$\+contramodules in
Lemma~\ref{separated-and-flat-contramodules-by-reduction-data}(b)
together with Lemma~\ref{flat-contramodules-well-behaved}(b).
 The fact that countable directed limits of diagrams of surjective
maps of abelian groups are exact functors needs to be used here.
 The general case of a countable base of right ideals is covered
by~\cite[Lemma~6.16]{PR} (see also~\cite[Proposition~6.17,
Remark~6.18, and Corollary~7.1(c)]{PR}
and~\cite[Proposition~3.10(c)]{Prev}).
\end{proof}

\begin{ex} \label{ind-affine-ind-scheme-example}
 In the context of algebraic geometry, a \emph{strict ind-affine\/
$\aleph_0$\+ind-scheme} $\fX$ is an ind-object in the category of
affine schemes representable by a countable directed diagram of closed
immersions of affine schemes $X_1\rarrow X_2\rarrow X_2\rarrow\dotsb$.
 So one has $X_n=\Spec R_n$ for some commutative rings $R_n$,
and the scheme morphisms $X_n\rarrow X_{n+1}$ correspond to surjective
ring homomorphisms $R_{n+1}\rarrow R_n$.

 The category of strict ind-affine $\aleph_0$\+ind-schemes is
equivalent to the category of complete, separated topological
commutative rings having a countable base of neighborhoods of zero
consisting of two-sided ideals.
 To an ind-scheme $\fX=\ilim_{n\ge1}X_n$ as above, the topological ring
$\R=\varprojlim_{n\ge1}R_n$ (with the topology of projective limit of
discrete rings~$R_n$) is assigned~\cite[Example~7.11.2(i)]{BD2},
\cite[Example~1.6(2)]{Psemten}.

 In this context, the sets of data from
Lemma~\ref{separated-and-flat-contramodules-by-reduction-data}(a)
are called ``$\cO^p$\+modules on~$\fX$'' in~\cite[Section~7.11.3]{BD2}
or \emph{pro-quasi-coherent pro-sheaves on\/~$\fX$}
in~\cite[Section~3.1]{Psemten}.
 The sets of data from
Lemma~\ref{separated-and-flat-contramodules-by-reduction-data}(b)
are called ``flat $\cO^p$\+modules on~$\fX$''
in~\cite[Section~7.11.3]{BD2} or \emph{flat pro-quasi-coherent
pro-sheaves on\/~$\fX$} in~\cite[Section~3.4]{Psemten}.
 So the category of pro-quasi-coherent pro-sheaves on $\fX$ is
equivalent to the category of separated $\R$\+contramodules, while
the category of flat pro-quasi-coherent pro-sheaves on $\fX$ is
equivalent to the category of flat $\R$\+contramodules.
 We refer to~\cite[Remark~7.11.3(ii)]{BD2}
and~\cite[Examples~3.1 and~3.8]{Psemten} for further discussion.
 For a discussion of non-ind-affine $\aleph_0$\+ind-schemes with further
details, see Remark~\ref{non-ind-affine-ind-schemes-remark} below.
\end{ex}

\Section{Countably Presentable Flat Contramodules}

 We refer to Section~\ref{accessible-preliminaries-secn} above for
the definitions of \emph{$\kappa$\+presentable objects},
\emph{locally $\kappa$\+presentable categories}, and
\emph{$\kappa$\+accessible categories},  with the references
to~\cite{AR}.

\begin{prop} \label{presentable-contramodules-described}
 Let $\kappa$ be an uncountable regular cardinal and\/ $\R$ be
a complete, separated topological ring with a base of neighborhoods
of zero of cardinality less than~$\kappa$, consisting of
open right ideals. \par
\textup{(a)} The category of left\/ $\R$\+contramodules\/
$\R\Contra$ is locally $\kappa$\+presentable.
 A left\/ $\R$\+contramodule\/ $\fP$ is $\kappa$\+presentable as
an object of the category\/ $\R\Contra$ if and only if it is
the cokernel of a morphism of free left\/ $\R$\+contramodules\/
$\R[[Y]]\rarrow\R[[X]]$ with the sets $X$ and $Y$ having cardinalities
less than~$\kappa$. \par
\textup{(b)} The category of complexes of left\/ $\R$\+contramodules\/
$\Com(\R\Contra)$ is locally $\kappa$\+presentable.
 A complex of left\/ $\R$\+contramodules\/ $\fP^\bu$ is 
$\kappa$\+presentable as an object of the category\/ $\Com(\R\Contra)$
if and only if all its terms\/ $\fP^n$, \,$n\in\boZ$, are
$\kappa$\+presentable\/ $\R$\+contramodules.
\end{prop}

\begin{proof}
 Part~(a): the ``if'' assertion is easily provable using the fact
that the forgetful functor $\R\Contra\rarrow\R\Modl$ preserves
$\kappa$\+directed colimits.
 The latter follows from the observation that any zero-convergent family
of elements in $\R$ has cardinality less than~$\kappa$.
 To prove the ``only if'' assertion, one then needs to check that all
left $\R$\+contramodules are $\kappa$\+directed colimits of
the cokernels of morphisms of free contramodules with less
than~$\kappa$ generators, and the class of all such cokernels is
closed under direct summands.
 The first assertion of part~(a) follows immediately (cf.\
the discussion in~\cite[Sections~1.1 and~5]{PR} and~\cite[Sections~6.2
and~6.4]{PS1}).

 Part~(b): the ``if'' assertion follows from the ``if'' assertion of
part~(a) by Lemma~\ref{complexes-presentability}(b).
 To deduce the ``only if'', one can check directly that all complexes of
left $\R$\+contramodules are $\kappa$\+directed colimits of complexes
whose terms are the cokernels of morphisms of free contramodules with
less than~$\kappa$ generators.
 The first assertion is then also clear (cf.~\cite[second paragraph
of the proof of Lemma~6.3]{PS4}).
 Alternatively, one can refer to Theorem~\ref{diagram-theorem}
or a suitable additive version of~\cite[Theorem~1.2]{Hen}, similarly
to the argument in Proposition~\ref{complexes-of-flat-comodules-prop}
and Remark~\ref{arbitrary-complexes-of-comodules}.
\end{proof}

 Now let $\R$ be a complete, separated topological ring with
a \emph{countable} base of neighborhoods of zero consisting of open
right ideals.
 Then the fully faithful inclusion functor $\R\Contra_\sep\rarrow
\R\Contra$ has a left adjoint functor $\Lambda_\R\:\R\Contra\rarrow
\R\Contra_\sep$ assigning to every left $\R$\+contramodule $\fP$
its maximal separated quotient $\R$\+contramodule
$$
 \Lambda_\R(\fP)=\varprojlim\nolimits_{\fI\subset\R}\fP/(\fI\tim\fP).
$$
 It is explained in~\cite[Lemma~6.2]{PR}
or~\cite[Proposition~4.2(b)]{Pcoun} how to endow the directed limit
$\varprojlim_{\fI\subset\R}\fP/(\fI\tim\fP)$ computed in the category
of abelian groups with a left $\R$\+contramodule structure.
 The natural morphism $\lambda_{\R,\fP}\:\fP\rarrow
\varprojlim_{\fI\subset\R}\fP/(\fI\tim\fP)$ is the adjunction unit.

\begin{lem} \label{max-separated-quotient-is-flat-implies-coincides}
 Let\/ $\R$ be a complete, separated topological ring with
a \emph{countable} base of neighborhoods of zero consisting of open
right ideals.
 Let\/ $\fP$ be a left\/ $\R$\+contramodule such that the separated
left\/ $\R$\+contramodule\/ $\Lambda_\R(\fP)$ is flat.
 Then the\/ $\R$\+contramodule\/ $\fP$ is flat and separated, and
the natural morphism\/ $\lambda_{\R,\fP}\:\fP\rarrow \Lambda_\R(\fP)$
is an isomorphism.
\end{lem}

\begin{proof}
 This is an intermediate step in the proofs
of~\cite[Corollary~E.1.7]{Pcosh} and~\cite[Corollary~6.15]{PR}
(which are restated above as Lemma~\ref{flat-contramodules-separated}).
 The argument is based on the contramodule Nakayama
lemma (Lemma~\ref{contramodule-nakayama} above).
 Alternatively, the same proof can be phrased as follows.
 One observes that, for any left $\R$\+contramodule $\fP$,
the adjunction morphism~$\lambda_{\R,\fP}$ induces an isomorphism
of the contratensor product functors ${-}\ocn_\R\fP\simeq
{-}\ocn_\R\Lambda_\R(\fP)$ (essentially, the nonseparated part of
$\fP$ is killed by contratensor products).
 This shows that flatness of $\Lambda_\R(\fP)$ implies flatness
of~$\fP$.
 Then the separatedness and completeness of $\fP$ follow from
Lemmas~\ref{all-contramodules-complete}
and~\ref{flat-contramodules-separated}.
\end{proof}

\begin{prop} \label{presentable-separated-contramodules-described}
 Let\/ $\R$ be a complete, separated topological ring with
a \emph{countable} base of neighborhoods of zero consisting of open
right ideals, and let $\kappa$~be an uncountable regular cardinal.
 Then the category of separated left\/ $\R$\+contramodules\/
$\R\Contra_\sep$ is locally $\kappa$\+presentable.
 A separated left\/ $\R$\+contramodule\/ $\fM$ is\/
$\kappa$\+pre\-sentable as an object of the category\/ $\R\Contra_\sep$
if and only if it has the form\/ $\fM=\Lambda_\R(\fP)$, where\/ $\fP$
is a $\kappa$\+presentable object of the category\/ $\R\Contra$.
\end{prop}

\begin{proof}
 One needs to observe that the full subcategory $\R\Contra_\sep
\subset\R\Contra$ is closed under $\aleph_1$\+directed colimits.
 Indeed, for any closed subgroup $\fA\subset\R$, the functor
$\fP\longmapsto\fA\tim\fP$ commutes with $\aleph_1$\+directed colimits
(since the functor $X\longmapsto\fA[[X]]\:\Sets\rarrow\Sets$ does),
and countable directed limits also commute with $\aleph_1$\+directed
colimits.
 Consequently, the inclusion functor $\R\Contra_\sep\rarrow\R\Contra$
preserves $\kappa$\+directed colimits, and it follows that the left
adjoint functor $\Lambda_\R\:\R\Contra\rarrow\R\Contra_\sep$ takes
$\kappa$\+presentable objects to $\kappa$\+presentable objects.
 This proves the ``if'' assertion.

 Furthermore, the functor $\Lambda_\R$, being a left adjoint, preserves
all colimits.
 Hence all colimits exist in $\R\Contra_\sep$ and can be constructed
by applying $\Lambda_\R$ to the colimits in $\R\Contra$.
 Since all objects of $\R\Contra$ are $\kappa$\+directed colimits
of $\kappa$\+presentable objects by
Proposition~\ref{presentable-contramodules-described}(a), so are
all the objects of $\R\Contra_\sep$.
 Thus the category $\R\Contra_\sep$ is locally $\kappa$\+presentable.

 To deduce the assertion ``only if'', we notice that all the objects of
$\R\Contra_\sep$ are $\kappa$\+directed colimits of separated
contramodules of the desired form $\fM=\Lambda_\R(\fP)$ with
$\fP\in\R\Contra_{<\kappa}$, as we have just shown.
 It remains to check that the class of all separated contramodules of
the desired form $\fM=\Lambda_\R(\fP)$ is closed under direct summands.
 Let us prove this class is closed under cokernels in $\R\Contra_\sep$.

 Suppose given a morphism $g\:\fM\rarrow\fN$, where
$\fM=\Lambda_\R(\fP)$, \,$\fN=\Lambda_\R(\fQ)$, and $\fP$,
$\fQ\in\R\Contra_{<\kappa}$.
 Then the natural (adjunction) morphisms $\fP\rarrow\fM$ and
$\fQ\rarrow\fN$ are surjective.
 Let $\fF\rarrow\fP$ be a surjective morphism onto $\fP$ from
a free left $\R$\+contramodule with less than~$\kappa$ generators.
 Then the composition $\fF\rarrow\fP\rarrow\fM\rarrow\fN$ can be lifted
to a morphism $f\:\fF\rarrow\fQ$.
 The cokernel of the morphism~$g$ in $\R\Contra_\sep$ can be obtained
by applying $\Lambda_\R$ to the cokernel of the morphism~$f$ in
$\R\Contra$.
 The latter cokernel is $\kappa$\+presentable in $\R\Contra$.
\end{proof}

\begin{lem} \label{presentable-flat-contramodules-lemma}
 Let\/ $\R$ be a complete, separated topological ring with
a \emph{countable} base of neighborhoods of zero consisting of open
right ideals, and let $\kappa$~be an uncountable regular cardinal.
 Then a \emph{flat} left\/ $\R$\+contramodule is $\kappa$\+presentable
as an object of the abelian category\/ $\R\Contra$ if and only if
it is $\kappa$\+presentable as an object of the category\/
$\R\Contra_\sep$.
\end{lem}

\begin{proof}
 The ``only if'' assertion holds, since the functor $\Lambda_\R\:
\R\Contra\rarrow\R\Contra_\sep$ takes $\kappa$\+presentable objects to
$\kappa$\+presentable objects, as explained in the proof of
Proposition~\ref{presentable-separated-contramodules-described},
and flat contramodules are separated by
Lemma~\ref{flat-contramodules-separated}.
{\emergencystretch=1em\par}

 The proof of the ``if'' is based on
Lemma~\ref{max-separated-quotient-is-flat-implies-coincides}
and the ``only if'' assertion of
Proposition~\ref{presentable-separated-contramodules-described}.
 If a left $\R$\+contramodule $\fF$ is
$\kappa$\+presentable in $\R\Contra_\sep$, then
Proposition~\ref{presentable-separated-contramodules-described} tells
us that $\fF=\Lambda_\R(\fP)$, where $\fP$ is $\kappa$\+presentable
in $\R\Contra$.
 Now if $\fF$ is flat, then $\fP\simeq\fF$ by
Lemma~\ref{max-separated-quotient-is-flat-implies-coincides}.
\end{proof}

\begin{lem} \label{prsntable-separated-and-flat-over-base-of-two-sided}
 Let\/ $\R$ be a complete, separated topological ring, and let\/
$\R\supset\fI_1\supset\fI_2\supset\fI_3\supset\dotsb$ be a countable
base of neighborhoods of zero in\/ $\R$ consisting of open
\emph{two-sided} ideals\/ $\fI_n\subset\R$, \,$n\ge1$.
 Let $\kappa$~be an uncountable regular cardinal.
 Then a separated\/ $\R$\+contramodule\/ $\fM$ is
$\kappa$\+presentable as an object of\/ $\R\Contra_\sep$ if and only if
the\/ $\R/\fI_n$\+module $M_n=\fM/(\fI_n\tim\fM)$ is\/
$\kappa$\+presentable for every $n\ge1$.
 Consequently, a flat\/ $\R$\+contramodule\/ $\fF$ is
$\kappa$\+presentable as an object of\/ $\R\Contra$ (or of\/
$\R\Contra_\sep$) if and only if the flat\/ $\R/\fI_n$\+module
$F_n=\fF/(\fI_n\tim\fF)$ is $\kappa$\+presentable for every $n\ge1$.
\end{lem}

\begin{proof}
 The first ``only if'' assertion follows easily from the ``only if''
assertion of
Proposition~\ref{presentable-separated-contramodules-described}.
 To prove the first ``if'', use the description of the category
$\R\Contra_\sep$ provided by
Lemma~\ref{separated-and-flat-contramodules-by-reduction-data}(a)
together with the fact that $\aleph_1$\+directed colimits commute with
countable limits in the category of abelian groups.
 Then the remaining assertions (concerning flat contramodules) follow
from Lemma~\ref{presentable-flat-contramodules-lemma}.
\end{proof}

\begin{lem} \label{projective-dimension-of-flat-contramodule}
 Let\/ $\R$ be a complete, separated topological ring, and let\/
$\R\supset\fI_1\supset\fI_2\supset\fI_3\supset\dotsb$ be a countable
base of neighborhoods of zero in\/ $\R$ consisting of open
\emph{two-sided} ideals\/ $\fI_n\subset\R$, \,$n\ge1$.
 Then the projective dimension of a flat\/ $\R$\+contramodule\/ $\fF$
(as an object of\/ $\R\Contra$) is equal to the supremum of
the projective dimensions of the flat\/ $\R/\fI_n$\+modules
$F_n=\fF/(\fI_n\tim\fF)$ (as objects of\/ $\R/\fI_n\Modl$), \,$n\ge1$.
\end{lem}

\begin{proof}
 Let $\fP_\bu\rarrow\fF$ be a projective resolution of $\fF$ in
$\R\Contra$ and $\fF_i$ be the image of the differential
$\fP_i\rarrow\fP_{i-1}$ (so $\fF=\fF_0$).
 Applying Lemma~\ref{flat-contramodules-well-behaved}(a) iteratively
and keeping in mind that projective contramodules are flat, one shows
that the $\R$\+contramodules $\fF_i$ are flat for all $i\ge0$.
 Now Lemma~\ref{flat-contramodules-well-behaved}(b) tells us that
the functors $\fQ\longmapsto(\R/\fI_n)\ocn_\R\fQ=\fQ/(\fI_n\tim\fQ)$
preserve exactness of the short exact sequences $0\rarrow\fF_i
\rarrow\fP_i\rarrow\fF_{i-1}\rarrow0$.
 Therefore, $(\R/\fI_n)\ocn_\R\fP_\bu$ is a projective resolution of
the $\R/\fI_n$\+module $F_n=(\R/\fI_n)\ocn_\R\fF$, and
the $\R/\fI_n$\+module $(\R/\fI_n)\ocn_\R\fF_i$ is the image of
the differential $(\R/\fI_n)\ocn_\R\fP_i\rarrow(\R/\fI_n)
\ocn_\R\fP_{i-1}$.
 It remains to use the characterization of projective
$\R$\+contramodules provided by
Lemma~\ref{projective-contramodules-by-reduction-data} in order to
deduce the desired assertion.
\end{proof}

\begin{cor} \label{count-pres-flat-contra-projdim1-cor}
 Let\/ $\R$ be a complete, separated topological ring with
a \emph{countable} base of neighborhoods of zero consisting of open
\emph{two-sided} ideals.
 Then any countably presentable flat\/ $\R$\+contramodule has projective
dimension at most\/~$1$ in the abelian category\/ $\R\Contra$.
\end{cor}

\begin{proof}
 Recall that all countably presentable flat modules (over a discrete
ring~$R$) have projective dimension at most~$1$ in $R\Modl$, by
Lemma~\ref{count-pres-flat-mods}.
 Then it remains to compare
Lemma~\ref{prsntable-separated-and-flat-over-base-of-two-sided}
(for flat contramodules and $\kappa=\aleph_1$) with
Lemma~\ref{projective-dimension-of-flat-contramodule}.
\end{proof}

\begin{qst} \label{count-pres-flat-contra-projdim-base-of-right-qst}
 One would expect the assertion of
Corollary~\ref{count-pres-flat-contra-projdim1-cor} to hold for left
contramodules over any complete, separated topological ring $\R$ with
a countable base of neighborhoods of zero consisting of open
\emph{right} ideals.
 But we do not know whether this is true.
 In fact, we don't even know the answer to
Question~\ref{contramodule-flat-coherence-base-of-right-qst} below,
which would seem to be easier.
\end{qst}

\Section{Flat Contramodules as Directed Colimits}

 The following theorem describing flat $\R$\+contramodules is the second
main result of this paper.
 It should be compared with the discussion of the conventional finite,
rather than countable Govorov--Lazard theorem for flat contramodules
in the paper~\cite[Section~10 and Example~13.4]{PPT}.
 The counterexamples in~\cite{PPT} show that the finite Govorov--Lazard
theorem fails for flat contramodules over topological rings
\emph{without} a countable base of neighborhoods of zero
(it is still an open question whether it holds under the assumption of
a countable base).
 Theorem~\ref{flat-contramodules-as-directed-colimits-theorem} is also
a contramodule analogue of~\cite[Theorems~2.4 and~3.5]{PS6}
and Theorem~\ref{flat-comodules-as-directed-colimits-theorem} above.

 We refer to Lemmas~\ref{presentable-flat-contramodules-lemma}
and~\ref{prsntable-separated-and-flat-over-base-of-two-sided} for
an explication of what is meant by a countably presentable flat
$\R$\+contramodule in
Theorem~\ref{flat-contramodules-as-directed-colimits-theorem}
(i.~e., countably presentable in $\R\Contra$ or $\R\Contra_\sep$, as
described in Propositions~\ref{presentable-contramodules-described}(a)
and~\ref{presentable-separated-contramodules-described}).

\begin{thm} \label{flat-contramodules-as-directed-colimits-theorem}
 Let\/ $\R$ be a complete, separated topological ring with
a \emph{countable} base of neighborhoods of zero consisting of open
\emph{two-sided} ideals.
 Then the category\/ $\R\Contra_\flat$ of flat left\/
$\R$\+contramodules is\/ $\aleph_1$\+accessible.
 The\/ $\aleph_1$\+presentable objects of\/ $\R\Contra_\flat$ are
precisely all the countably presentable flat\/ $\R$\+contramodules.
 Consequently, every flat\/ $\R$\+contramodule is
an\/ $\aleph_1$\+directed colimit of countably presentable flat\/
$\R$\+contramodules.
\end{thm}

\begin{proof}
 The full subcategory $\R\Contra_\flat$ is closed under directed
colimits in $\R\Contra$.
 Therefore, all countably presentable flat $\R$\+contramodules are
$\aleph_1$\+pre\-sentable in $\R\Contra_\flat$.
{\hbadness=1400\par}

 The full assertion of the theorem is obtained by applying
Theorem~\ref{isomorpher-theorem} for $\kappa=\aleph_1$ and
$\lambda=\aleph_0$.
 The argument is based on the description of the category of
flat $\R$\+contramodules provided by
Lemma~\ref{separated-and-flat-contramodules-by-reduction-data}(b)
and the description of countably presentable separated/flat
$\R$\+countramodules provided by
Lemma~\ref{prsntable-separated-and-flat-over-base-of-two-sided}.

 Let $\R\supset\fI_1\supset\fI_2\supset\fI_3\supset\dotsb$ be
a countable base of neighborhoods of zero in $\R$ consisting of open
two-sided ideals.
 Let $\sK=\sL$ be the Cartesian product of the categories of flat left
modules over the rings $\R/\fI_n$, taken over the integers $n\ge1$;
so $\sK=\sL=\prod_{n=1}^\infty(\R/\fI_n\Modl_\flat)$.
 Consider the following pair of parallel functors $F_1$, $F_2\:
\sK\rightrightarrows\sL$.
 The functor $F_1$ takes a collection of flat modules
$(K_n\in\R/\fI_n\Modl_\flat)_{n\ge1}$ to the collection of
flat modules $L_n=\R/\fI_n\ot_{\R/\fI_{n+1}}K_{n+1}$.
 The functor $F_2$ is the identity functor.

 Then Lemma~\ref{separated-and-flat-contramodules-by-reduction-data}(b)
tells us that the isomorpher category $\sC$ is equivalent to
the category of flat left $\R$\+contramodules $\R\Contra_\flat$.
 The categories $\R/\fI_n\Modl_\flat$ are $\aleph_1$\+accessible
by Lemma~\ref{flat-modules-accessible}, and their Cartesian product
$\prod_{n=1}^\infty(\R/\fI_n\Modl_\flat)$ is $\aleph_1$\+accessible
by Proposition~\ref{product-proposition}.
 Theorem~\ref{isomorpher-theorem} is applicable, and it tells us that
the category $\sC$ is $\aleph_1$\+accessible.
 Comparing the description of the $\aleph_1$\+presentable objects of
$\sC$ provided by Theorem~\ref{isomorpher-theorem} with the description
of the $\aleph_1$\+presentable flat $\R$\+contramodules provided by
Lemma~\ref{prsntable-separated-and-flat-over-base-of-two-sided} one 
can see that these two classes of objects coincide, as desired.
\end{proof}

 Let us also provide a description of arbitrary complexes of flat
$\R$\+contramodules as directed colimits.

\begin{prop} \label{complexes-of-flat-contramodules-prop}
 Let\/ $\R$ be a complete, separated topological ring with
a \emph{countable} base of neighborhoods of zero consisting of open
\emph{two-sided} ideals.
 Then the category\/ $\Com(\R\Contra_\flat)$ of complexes of flat
left\/ $\R$\+contramodules is\/ $\aleph_1$\+accessible.
 The\/ $\aleph_1$\+presentable objects of\/ $\Com(\R\Contra_\flat)$ are
precisely all the complexes of countably presentable flat\/
$\R$\+contramodules.
 Consequently, every complex of flat\/ $\R$\+contramodules is
an\/ $\aleph_1$\+directed colimit of complexes of countably
presentable flat\/ $\R$\+contramodules.
\end{prop}

\begin{proof}
 All complexes of countably presentable $\R$\+contramodules are
$\aleph_1$\+presentable as objects of $\Com(\R\Contra)$ by
Proposition~\ref{presentable-contramodules-described}(b).
 Since the full subcategory $\Com(\R\Contra_\flat)$ is closed under
directed colimits in $\Com(\R\Contra)$, it follows that all complexes
of countably presentable flat $\R$\+contramodules are
$\aleph_1$\+presentable in $\Com(\R\Contra_\flat)$.

 The full assertion of the proposition is obtained by combining
the results of
Theorems~\ref{flat-contramodules-as-directed-colimits-theorem}
and~\ref{diagram-theorem}.
 The argument is similar to the proof of
Proposition~\ref{complexes-of-flat-comodules-prop}.
\end{proof}

\begin{rem} \label{non-ind-affine-ind-schemes-remark}
 One can restrict the results of this section to commutative topological
rings $\R$, interpret commutative topological rings as ind-affine
ind-schemes (as per Example~\ref{ind-affine-ind-scheme-example}), and
then generalize to non-ind-affine ind-schemes.
 To avoid a detailed discussion of the basics of the ind-scheme theory
(see~\cite[Section~1.2]{Psemten} and the references therein), let us
say that a \emph{strict ind-quasi-compact ind-quasi-separated\/
$\aleph_0$\+ind-scheme} $\fX$ can be defined as an ind-object in
the category of quasi-compact quasi-separated schemes representable by
a countable directed diagram of closed immersions of quasi-compact
quasi-separated schemes $X_1\rarrow X_2\rarrow X_3\rarrow\dotsb$.

 An ``$\cO^p$\+module on~$\fX$'' (in the terminology
of~\cite[Section~7.11.3]{BD2}) or a \emph{pro-quasi-coherent pro-sheaf
on\/~$\fX$} in our preferred terminology of~\cite[Section~3.1]{Psemten}
can be then defined as a sequence of quasi-coherent sheaves $M_n$ of
the schemes $X_n$ together with isomorphisms of quasi-coherent sheaves
$M_n\simeq i_n^*M_{n+1}$ on $X_n$, where $i_n\:X_n\rarrow X_{n+1}$
denote the closed immersion morphisms ($n\ge1$).
 A pro-quasi-coherent pro-sheaf $(F_n)_{n\ge1}$ is said to be
\emph{flat} if the quasi-coherent sheaf $F_n$ on $X_n$ is flat
for every $n\ge1$.
 The flat pro-quasi-coherent pro-sheaves form a well-behaved (exact)
full subcategory in a badly behaved additive category of arbitrary
pro-quasi-coherent pro-sheaves on $\fX$ (see
Example~\ref{ind-affine-ind-scheme-example} and
Remark~\ref{separated-contramodules-not-well-behaved}).

 The argument similar to the proof of
Theorem~\ref{flat-contramodules-as-directed-colimits-theorem}
and based on the result of~\cite[Theorem~2.4]{PS6} shows that
the category of flat pro-quasi-coherent pro-sheaves on $\fX$ is
$\aleph_1$\+accessible, and the flat pro-quasi-coherent
pro-sheaves $(F_n)_{n\ge1}$ with locally countably presentable
flat quasi-coherent sheaves $F_n$ on $X_n$ are
the $\aleph_1$\+presentable objects of this category.
 The flat pro-quasi-coherent pro-sheaf version of
Proposition~\ref{complexes-of-flat-contramodules-prop} also holds,
with the same proof.
 The same applies to the results of
Proposition~\ref{short-exact-sequences-of-flat-contramodules} and
Corollary~\ref{contramods-pure-acycl-cplxs-as-aleph1-dir-colims} below.
\end{rem}

\Section{Exact Sequences of Flat Contramodules as Directed Colimits}

 Let us start with a contramodule analogue/generalization of
Lemma~\ref{countable-flat-coherence}.

\begin{lem} \label{contramodule-countable-flat-coherence}
 Let\/ $\R$ be a complete, separated topological ring with
a \emph{countable} base of neighborhoods of zero consisting of open
\emph{two-sided} ideals.
 Then the kernel of any surjective morphism from a countably
presentable flat\/ $\R$\+contramodule to a countably presentable flat\/
$\R$\+contramodule is a countably presentable (flat)
$\R$\+contramodule.
\end{lem}

\begin{proof}
 The class of flat $\R$\+contramodules is closed under kernels of
surjective morphisms in $\R\Contra$ by
Lemma~\ref{flat-contramodules-well-behaved}(a).
 Now if $0\rarrow\fH\rarrow\fG\rarrow\fF\rarrow0$ is a short exact
sequence of flat left $\R$\+contramodules, then
Lemma~\ref{flat-contramodules-well-behaved}(b) tells us that
$0\rarrow(\R/\fI)\ocn_\R\fH\rarrow(\R/\fI)\ocn_\R\fG\rarrow
(\R/\fI)\ocn_\R\fF\rarrow0$ is a short exact sequence of
left $\R/\fI$\+modules for any open two-sided ideal $\fI\subset\R$.
 If the $\R$\+contramodules $\fF$ and $\fG$ are countably presentable,
then so are the $\R/\fI$\+modules $(\R/\fI)\ocn_\R\fF$ and
$(\R/\fI)\ocn_\R\fG$, by
Lemma~\ref{prsntable-separated-and-flat-over-base-of-two-sided}.
 So Lemma~\ref{countable-flat-coherence} tells us that
the $\R/\fI$\+module $(\R/\fI)\ocn_\R\fH$ is countably presentable,
and it remains to apply
Lemma~\ref{prsntable-separated-and-flat-over-base-of-two-sided} again
in order to conclude that the $\R$\+contramodule $\fH$ is
countably presentable.

 Alternatively, one can deduce the lemma from
Corollary~\ref{count-pres-flat-contra-projdim1-cor} essentially in
the same way as the proof of Lemma~\ref{countable-flat-coherence}
deduces it from Lemma~\ref{count-pres-flat-mods}.
 Then one needs to use the facts that the kernel of any surjective
morphism of countably presentable $\R$\+contramodules is a countably
generated $\R$\+contramodule, and any countably generated projective
$\R$\+contramodule is countably presentable.
\end{proof}

\begin{qst} \label{contramodule-flat-coherence-base-of-right-qst}
 Similarly to
Question~\ref{count-pres-flat-contra-projdim-base-of-right-qst},
we do \emph{not} know, and it would be interesting to learn, whether
the assertion of Lemma~\ref{contramodule-countable-flat-coherence}
holds for left contramodules over any complete, separated topological
ring $\R$ with a countable base of neighborhoods of zero consising
of open \emph{right} ideals.
\end{qst}

 The following proposition describes short exact sequences of
flat $\R$\+contramodules.

\begin{prop} \label{short-exact-sequences-of-flat-contramodules}
 Let\/ $\R$ be a complete, separated topological ring with
a \emph{countable} base of neighborhoods of zero consisting of open
\emph{two-sided} ideals.
 Then the category of short exact sequences of flat\/
$\R$\+contramodules is\/ $\aleph_1$\+accessible.
 The\/ $\aleph_1$\+presentable objects of this category are precisely
all the short exact sequences of countably presentable flat\/
$\R$\+contramodules.
 Consequently, every short exact sequence of flat\/ $\R$\+contramodules
is an\/ $\aleph_1$\+directed colimit of short exact sequences of
countably presentable flat\/ $\R$\+contramodules.
\end{prop}

\begin{proof}
 The argument is similar to the proof of
Theorem~\ref{flat-contramodules-as-directed-colimits-theorem} and
uses Lemmas~\ref{short-exact-of-flat-modules-accessible}
and~\ref{flat-contramodules-well-behaved}(b) together with
Lemmas~\ref{separated-and-flat-contramodules-by-reduction-data}(b)
and~\ref{prsntable-separated-and-flat-over-base-of-two-sided}.
 The assertion is obtained by applying
Proposition~\ref{product-proposition} and
Theorem~\ref{isomorpher-theorem} for $\kappa=\aleph_1$ and
$\lambda=\aleph_0$.

 Let $\R\supset\fI_1\supset\fI_2\supset\fI_3\supset\dotsb$ be
a countable base of neighborhoods of zero in $\R$ consisting of open
two-sided ideals.
 Let $\sK=\sL$ be the Cartesian product of the categories of short
exact sequences of flat left modules over the rings $\R/\fI_n$,
taken over the integers $n\ge1$.
 Consider the pair of functors $F_1$, $F_2\:\sK\rightrightarrows\sL$
similar to the one in the proof of
Theorem~\ref{flat-contramodules-as-directed-colimits-theorem}.

 Then Lemmas~\ref{flat-contramodules-well-behaved}(b)
and~\ref{separated-and-flat-contramodules-by-reduction-data}(b)
imply that the isomorpher category $\sC$ is equivalent to the category
of short exact sequences of flat left $\R$\+contramodules (one also
needs to use the fact that directed limits of countable sequences of
surjective maps of short exact sequences of abelian groups are short
exact sequences of abelian groups again).
 The categories $\sK$ and $\sL$ are $\aleph_1$\+accessible by
Lemma~\ref{short-exact-of-flat-modules-accessible} and
Proposition~\ref{product-proposition}.
 Theorem~\ref{isomorpher-theorem} is applicable; it tells us that
the category $\sC$ is $\aleph_1$\+accessible and provides a description
of its full subcategory of $\aleph_1$\+presentable objects.
 To compare it with the description asserted in the proposition,
one needs to use
Lemma~\ref{prsntable-separated-and-flat-over-base-of-two-sided}.
\end{proof}

 The assertion of
Proposition~\ref{short-exact-sequences-of-flat-contramodules} (together
with Theorem~\ref{flat-contramodules-as-directed-colimits-theorem})
can be restated by saying that the exact category of flat
$\R$\+contramodules $\R\Contra_\flat$ is
a \emph{locally\/ $\aleph_1$\+coherent exact category} in
the sense of~\cite[Section~1]{Plce}.

 Let us say that an acyclic complex of flat $\R$\+contramodules is
\emph{pure acyclic} if its $\R$\+contramodules of cocycles
are flat.
 For topological rings $\R$ with a countable base of neighborhoods
of zero, this is consistent with the notion of \emph{contratensor
purity}~\cite[Section~3]{Pproperf}, \cite[Section~13]{PPT}.
 It is worth noticing that any bounded above acyclic complex of flat
contramodules is pure acyclic by
Lemma~\ref{flat-contramodules-well-behaved}(a).

\begin{cor} \label{contramods-pure-acycl-cplxs-as-aleph1-dir-colims}
 Let\/ $\R$ be a complete, separated topological ring with
a \emph{countable} base of neighborhoods of zero consisting of open
\emph{two-sided} ideals.
 Then the category of pure acyclic complexes of flat\/
$\R$\+contramodules is\/ $\aleph_1$\+accessible.
 The\/ $\aleph_1$\+presentable objects of this category are precisely
all the pure acyclic complexes of countably presentable flat\/
$\R$\+contramodules.
 Consequently, every pure acyclic complex of flat\/ $\R$\+contramodules
is an\/ $\aleph_1$\+directed colimit of pure acyclic complexes of
countably presentable flat\/ $\R$\+contramodules.
\end{cor}

\begin{proof}
 Similar to the proof of
Corollary~\ref{comods-pure-acycl-complexes-as-aleph1-dir-colims}.
 One can either prove the assertion of the corollary by a direct
argument similar to the proof of
Proposition~\ref{short-exact-sequences-of-flat-contramodules}
and using~\cite[Corollary~10.14]{Pacc} instead of
Lemma~\ref{short-exact-of-flat-modules-accessible}, or deduce
the corollary from
Proposition~\ref{short-exact-sequences-of-flat-contramodules}
using the argument with building pure acyclic complexes by splicing
short exact sequences as in the proof of
Corollary~\ref{comods-pure-acycl-complexes-as-aleph1-dir-colims}.
\end{proof}

\Section{Cotorsion Periodicity for Contramodules}

 For the sake of completeness of the exposition, we start with
presenting a weak version of flat/projective periodicity
for contramodules.
 The definition of a \emph{periodic object} (with respect to a class
of objects in an abelian category) was given in
Section~\ref{comodule-cotorsion-periodicity-secn}.
 The following proposition is a contramodule generalization
of~\cite[Theorem~2.5]{BG} and~\cite[Remark~2.15]{Neem}
(see also~\cite[Proposition~7.6]{CH}).
 For a strong version of flat/projective periodicity theorem for
contramodules (with a proof based on the results of this paper)
see the recent preprint~\cite{Pbc}.

\begin{prop} \label{contramodule-flat-projective-periodicity}
 Let\/ $\R$ be a complete, separated topological ring with
a \emph{countable} base of neighborhoods of zero consisting of open
\emph{two-sided} ideals.
 Then \par
\textup{(a)} any projective-periodic flat\/ $\R$\+contramodule is
projective;
\par
\textup{(b)} in any acyclic complex of projective\/ $\R$\+contramodules
with flat\/ $\R$\+contramodules of cocycles, the contramodules of
cocycles are actually projective (so the complex is contractible).
\end{prop}

\begin{proof}
 This is an easy corollary of the the flat/projective periodicity for
modules over a ring (\cite[Theorem~2.5]{BG}, \cite[Remark~2.15]{Neem})
together with the basic properties of flat and projective contramodules
(Lemmas~\ref{flat-contramodules-well-behaved}(b)
and~\ref{projective-contramodules-by-reduction-data}).

 Part~(a): let $0\rarrow\fF\rarrow\fP\rarrow\fF\rarrow0$ be a short
exact sequence of flat left $\R$\+contramodules with a projective
$\R$\+contramodule~$\fP$.
 By Lemma~\ref{flat-contramodules-well-behaved}(b), for any open
two-sided ideal $\fI\subset\R$, the short sequence of left
$\R/\fI$\+modules $0\rarrow(\R/\fI)\ocn_\R\fF\rarrow
(\R/\fI)\ocn_\R\fP\rarrow(\R/\fI)\ocn_\R\fF\rarrow0$ is exact.
 Clearly, the $\R/\fI$\+module $(\R/\fI)\ocn_\R\fF$ is flat and
the $\R/\fI$\+module $(\R/\fI)\ocn_\R\fP$ is projective
(see the discussion in
Sections~\ref{contra-preliminaries-secn}\+-%
\ref{separated-and-flat-contra-secn}, cf.\
Lemmas~\ref{separated-and-flat-contramodules-by-reduction-data}(b)
and~\ref{projective-contramodules-by-reduction-data}).
 By~\cite[Theorem~2.5]{BG}, it follows that
the $\R/\fI$\+module $(\R/\fI)\ocn_\R\fF$ is projective for every
open two-sided ideal $\fI\subset\R$.
 Using Lemma~\ref{projective-contramodules-by-reduction-data},
we conclude that the $\R$\+contramodule~$\fF$ is projective.

 Part~(b): let $\fP^\bu$ be a pure acyclic complex of projective
left $\R$\+contramodules.
 By Lemma~\ref{flat-contramodules-well-behaved}(b), for any open
two-sided ideal $\fI\subset\R$, the complex of left
$\R/\fI$\+modules $(\R/\fI)\ocn_\R\fP^\bu$ is (pure) acyclic.
 Moreover, if $\fZ^n$, \,$n\in\boZ$, are the $\R$\+contramodules of
cocycles of the complex $\fP^\bu$, then $(\R/\fI)\ocn_\R\fZ^n$ are
the $\R/\fI$\+modules of cocycles of the complex
$(\R/\fI)\ocn_\R\fP^\bu$.
 The complex $(\R/\fI)\ocn_\R\fP^\bu$ is also a complex of
projective left $\R/\fI$\+modules.
 By~\cite[Theorem~8.6 and Remark~2.15]{Neem}
or~\cite[Proposition~7.6]{CH}, it follows that the complex
$(\R/\fI)\ocn_\R\fP^\bu$ is contractible and the $\R/\fI$\+modules
$(\R/\fI)\ocn_\R\fZ^n$ are projective.
 Once again, we use
Lemma~\ref{projective-contramodules-by-reduction-data} in order to
conclude that the $\R$\+contramodules $\fZ^n$ are projective and
the complex $\fP^\bu$ is contractible.
\end{proof}

 Now we pass to the cotorsion periodicity for $\R$\+contramodules,
which is the main topic of this section.
 Let us denote by $\Ext^{\R,*}({-},{-})$ the Ext groups in
the abelian category $\R\Contra$.
 A left $\R$\+contramodule $\fB$ is said to be \emph{cotorsion}
\cite[Section~E.5]{Pcosh}, \cite[Definition~7.3]{PR} if
$\Ext^{\R,1}(\fF,\fB)=0$ for all flat left $\R$\+contramodules~$\fF$.

 Introduce the notation $\sB=\R\Contra^\cot$ for the full subcategory
of cotorsion $\R$\+contramodules in the abelian category
$\sC=\R\Contra$.
 Obviously, the full subcategory $\R\Contra^\cot$ is closed under
extensions in $\R\Contra$, so it inherits an exact category structure
from the abelian exact structure of $\R\Contra$.

 The following corollary summarizes some of our previous results and
observations.
 It is a contramodule version of~\cite[Corollary~5.6 and Lemma~9.1]{PS6}
and Corollary~\ref{prepare-for-comodule-cotorsion-periodicity} above.

\begin{cor} \label{prepare-for-contramodule-cotorsion-periodicity}
 Let\/ $\R$ be a complete, separated topological ring. \par
\textup{(a)} Assume that\/ $\R$ has a countable base of neighborhoods
of zero consisting of open right ideals.
 Then the class all flat left\/ $\R$\+contramodules is resolving
and closed under directed colimits in\/ $\R\Contra$.
 Consequently, one has\/ $\Ext^{\R,n}(\fF,\fB)=0$ for all flat
left\/ $\R$\+contramodules\/ $\fF$, all cotorsion left\/
$\R$\+contramodules\/ $\fB$, and all integers $n\ge1$.
 The directed colimit functors in the exact category of flat
contramodules\/ $\R\Contra_\flat$ are exact (i.~e., the directed
colimits of admissible short exact sequences are admissible
short exact sequences). \par
\textup{(b)} Assume that\/ $\R$ has a countable base of neighborhoods
of zero consisting of open two-sided ideals.
 Then any flat\/ $\R$\+contramodule is a directed (in fact,
$\aleph_1$\+directed) colimit of flat\/ $\R$\+contramodules having
projective dimensions not exceeding\/~$1$ in\/ $\R\Contra$.
\end{cor}

\begin{proof}
 Part~(a): the full subcategory $\R\Contra_\flat\subset\R\Contra$ is
resolving by Lemma~\ref{flat-contramodules-well-behaved}(a) and
closed under directed colimits according to the discussion in
Sections~\ref{contra-preliminaries-secn}\+-%
\ref{separated-and-flat-contra-secn}.
 The Ext vanishing assertion follows by~\cite[Lemma~6.1]{PS6}.
 The directed colimits in the exact category $\R\Contra_\flat$ (though
not in the abelian category $\R\Contra$\,!) are exact by
Lemma~\ref{directed-colimits-of-flat-contramodules-exact}.

 In part~(b), all flat $\R$\+contramodules are $\aleph_1$\+directed
colimits of countably presentable flat $\R$\+contramodules by
Theorem~\ref{flat-contramodules-as-directed-colimits-theorem}, and
countably presentable flat $\R$\+contramodules have projective dimension
at most~$1$ by Corollary~\ref{count-pres-flat-contra-projdim1-cor}.
\end{proof}

 The following result is a contramodule generalization of
the cotorsion periodicity theorem of Bazzoni, Cort\'es-Izurdiaga, and
Estrada~\cite[Theorem~1.2(2) or Proposition~4.8(2)]{BCE}.
 It is also a contramodule analogue of~\cite[Theorem~9.2]{PS6} and
Theorem~\ref{comodule-cotorsion-periodicity-theorem} above.

\begin{thm} \label{contramodule-cotorsion-periodicity-theorem}
 Let\/ $\R$ be a complete, separated topological ring with
a \emph{countable} base of neighborhoods of zero consisting of open
\emph{two-sided} ideals.
 Then any cotorsion-periodic $\R$\+contramodule is cotorsion.
\end{thm}

\begin{proof}
 This is another application of
Theorem~\ref{general-cotorsion-periodicity-scheme}, whose assumptions
are satisfied in the situation at hand (for $\sF=\R\Contra_\flat
\subset\R\Contra=\sC$) by
Corollary~\ref{prepare-for-contramodule-cotorsion-periodicity}.
\end{proof}

\begin{cor} \label{contramodule-cotorsion-cocycles-cor}
 Let\/ $\R$ be a complete, separated topological ring with
a \emph{countable} base of neighborhoods of zero consisting of open
\emph{two-sided} ideals.
 Let\/ $\fB^\bu$ be an acyclic complex in\/ $\R\Contra$ whose terms\/
$\fB^n$ are cotorsion\/ $\R$\+contramodules for all $n\in\boZ$.
 Then the\/ $\R$\+contramodules of cocycles of the complex\/ $\fB^\bu$
are also cotorsion.
\end{cor}

\begin{proof}
 This is a contramodule generalization of~\cite[Theorem~5.1(2)]{BCE},
and an analogue of~\cite[Corollary~9.4]{PS6} and
Corollary~\ref{comodule-cotorsion-cocycles-cor} above.
 The infinite product functors are exact in $\R\Contra$, and the class
of cotorsion $\R$\+contramodules is closed under (direct summands and)
infinite products.
 So the argument of~\cite[Proposition~2]{EFI} is applicable, deducing
the assertion of the corollary from
Theorem~\ref{contramodule-cotorsion-periodicity-theorem}.
\end{proof}

 The definition of a \emph{coresolving subcategory} is dual to that
of a resolving one.
 A class of objects $\sB$ in an abelian category $\sC$ is said to be
\emph{cogenerating} if every object of $\sC$ is a subobject of
an object from~$\sB$.
 We will say that a class of objects $\sB\subset\sC$ is
\emph{coresolving} if it is cogenerating, closed under extensions,
and closed under cokernels of monomorphisms in~$\sC$.

\begin{prop} \label{cotorsion-coresolving-prop}
 Let\/ $\R$ be a complete, separated topological ring with
a \emph{countable} base of neighborhoods of zero consisting of open
right ideals.
 Then the class\/ $\sB=\R\Contra^\cot$ of all cotorsion left\/
$\R$\+contramodules is coresolving in\/ $\sC=\R\Contra$.
\end{prop}

\begin{proof}
 It is obvious from the definition that the class of all cotorsion
$\R$\+contramodules is closed under extensions, and it follows from
the higher Ext vanishing assertion of
Corollary~\ref{prepare-for-contramodule-cotorsion-periodicity}(a)
that $\R\Contra^\cot$ is also closed under cokernels of monomorphisms
in $\R\Contra$.
 The most nontrivial assertion in the proposition is that the class
of cotorsion $\R$\+contramodules is cogenerating.
 This is a part of~\cite[Corollary~7.8]{PR}.
 (Notice that there are usually \emph{not} enough injective objects
in $\R\Contra$.)
\end{proof}

 Recall the notation $\sK(\sE)$ for the homotopy category of
an additive category~$\sE$.
 For an abelian (or exact) category $\sE$, we denote by $\Ac(\sE)
\subset\sK(\sE)$ the full subcategory of acyclic complexes.
 So the derived category $\sD(\sE)$ is the Verdier quotient category
$\sD(\sE)=\sK(\sE)/\Ac(\sE)$.

\begin{prop} \label{abelian-exact-products-coresolution-prop}
 Let\/ $\sC$ be an abelian category with exact functors of countable
product, and let\/ $\sB\subset\sC$ be a cogenerating subcategory closed
under countable products.
 Then for any complex $C^\bu$ in\/ $\sC$ there exists a complex $B^\bu$
in\/ $\sB$ together with a quasi-isomorphism $C^\bu\rarrow B^\bu$ of
complexes in\/~$\sC$.
 Consequently, the inclusion of additive categories\/ $\sB\rarrow\sC$
induces a triangulated equivalence of Verdier quotient categories
$$
 \frac{\sK(\sB)}{\sK(\sB)\cap\Ac(\sC)}\overset{\simeq}\lrarrow
 \frac{\sK(\sC)}{\Ac(\sC)}\,=\,\sD(\sC).
$$
\end{prop}

\begin{proof}
 The proof of the first assertion is similar to the construction
of $K$\+projective and $K$\+injective resolutions of complexes of
modules in~\cite[Section~3]{Spal}.
 Let us spell out the construction as follows.
 For every $n\in\boZ$, consider the canonical truncation
$\tau_{\ge n}C^\bu$ of the complex~$C^\bu$.
 Since the complex $\tau_{\ge n}C^\bu$ is bounded below and
the subcategory $\sB\subset\sC$ is cogenerating, one can easily find
a (bounded below) complex ${}^{(n)}D^\bu$ in $\sB$ together with
a quasi-isomorphism $\tau_{\ge n}C^\bu\rarrow{}^{(n)}D^\bu$ in~$\sC$.
 Moreover, one can choose the complex ${}^{(n)}D^\bu$ so that
the morphism of complexes $\tau_{\ge n}C^\bu\rarrow{}^{(n)}D^\bu$
is a termwise monomorphism.
 The composition $C^\bu\rarrow\tau_{\ge n}C^\bu\rarrow{}^{(n)}D^\bu$
provides a morphism of complexes $C^\bu\rarrow{}^{(n)}D^\bu$.

 Now the induced morphism into the product $C^\bu\rarrow B^{0,\bu}=
\prod_{n\in\boZ}{}^{(n)}D^\bu$ is a termwise monomorphism of complexes
in $\sC$ inducing monomorphisms on all the cohomology objects.
 Furthermore, $B^{0,\bu}$ is a complex in~$\sB$.
 Next we apply the same construction to the cokernel of the morphism
of complexes $C^\bu\rarrow B^{0,\bu}$, embedding the complex
$B^{0,\bu}/C^\bu$ into a complex $B^{1,\bu}$ in $\sB$ in such a way
that the morphism of complexes $B^{0,\bu}/C^\bu\rarrow B^{1,\bu}$ is
(not only termwise monic, but also) induces monomorphisms on all
the cohomology objects.
 Iterating the procedure, we obtain an exact complex of complexes
$$
 0\lrarrow C^\bu\lrarrow B^{0,\bu}\lrarrow B^{1,\bu}\lrarrow
 B^{2,\bu}\lrarrow\dotsb
$$
such that passing to the cohomology objects produces an exact complex
$$
 0\lrarrow H^*(C^\bu)\lrarrow H^*(B^{0,\bu})\lrarrow
 H^*(B^{1,\bu})\lrarrow H^*(B^{2,\bu})\lrarrow\dotsb
$$
of graded objects in~$\sC$.

 Finally, we totalize the bicomplex $B^{\bu,\bu}$ by taking infinite
products along the diagonals.
 Put $B^\bu=\Tot^\sqcap(B^{\bu,\bu})$ to be the product totalization.
 Then $B^\bu$ is a complex in~$\sB$.
 Furthermore, the resulting morphism of complexes $C^\bu\rarrow B^\bu$
is a quasi-isomorphism by the next Lemma~\ref{eilenberg-moore}.
 This proves the first assertion of the proposition, and the second
one then follows by~\cite[Lemma~9.5]{PS6}.
 An opposite version of this argument can be found in~\cite[proof of
Proposition~A.4.3]{Pcosh}.
\end{proof}

 The next lemma is a classical result of Eilenberg and Moore~\cite{EM}.

\begin{lem} \label{eilenberg-moore}
 Let\/ $\sC$ be an abelian category with exact functors of countable
product, and let
$$
 0\lrarrow C^{0,\bu}\lrarrow C^{1,\bu}\lrarrow
 C^{2,\bu}\lrarrow C^{3,\bu}\lrarrow\dotsb
$$
be a bounded below complex of complexes in\/~$\sC$.
 Assume that, passing to the cohomology objects, we obtain an acyclic
complex
$$
 0\lrarrow H^*(C^{0,\bu})\lrarrow H^*(C^{1,\bu})\lrarrow H^*(C^{2,\bu})
 \lrarrow H^*(C^{3,\bu})\lrarrow\dotsb
$$
of graded objects in\/~$\sC$.
 Then the product totalization\/ $\Tot^\sqcap(C^{\bu,\bu})$ of
the bicomplex $C^{\bu,\bu}$ is an acyclic complex in\/~$\sC$.
\end{lem}

\begin{proof}
 This is a particular case of~\cite[Theorem~7.4]{EM} or (essentially)
the dual version of~\cite[Lemma~A.4.4]{Pcosh}.
\end{proof}

 The following corollary is a contramodule/ind-affine ind-scheme
analogue of~\cite[Corollary~9.7]{PS6} and
Corollary~\ref{comodule-cotorsion-derived-equivalence} above.

\begin{cor} \label{contramodule-cotorsion-derived-equivalence}
 Let\/ $\R$ be a complete, separated topological ring with
a \emph{countable} base of neighborhoods of zero consisting of open
\emph{two-sided} ideals.
 Then the inclusion of exact/abelian categories\/ $\R\Contra^\cot
\rarrow\R\Contra$ induces an equivalence of their unbounded derived
categories,
$$
 \sD(\R\Contra^\cot)\,\simeq\sD(\R\Contra).
$$
\end{cor}

\begin{proof}
 It suffices to compare
Propositions~\ref{cotorsion-coresolving-prop}
and~\ref{abelian-exact-products-coresolution-prop} with
Corollary~\ref{contramodule-cotorsion-cocycles-cor}.
\end{proof}

\Section{Exact Sequences of Flat Contramodules~II}

 In this short section we present contramodule/ind-affine ind-scheme
analogues of the results of
Section~\ref{comodules-exact-sequences-II-secn}.
 We refer to Section~\ref{comodules-exact-sequences-II-secn} for
the discussion of \emph{self-resolving subcategories} in abelian/exact
categories.

 The following lemma is a contramodule analogue of
Corollary~\ref{comodules-count-pres-flat-exact-category}.

\begin{lem} \label{contramodules-count-pres-flat-exact-category}
 Let\/ $\R$ be a complete, separated topological ring with
a \emph{countable} base of neighborhoods of zero consisting of open
\emph{two-sided} ideals.
 Then \par
\textup{(a)} the full subcategory\/ $\sG$ of all countably presentable
flat\/ $\R$\+contramodules is self-resolving in the abelian category\/
$\R\Contra$; \par
\textup{(b)} the homological dimension of the exact category\/ $\sG$
(with the exact structure inherited from\/ $\R\Contra$) does not
exceed\/~$1$.
\end{lem}

\begin{proof}
 Part~(a): the class of all flat $\R$\+contramodules is closed under
extensions in $\R\Contra$ by
Lemma~\ref{flat-contramodules-well-behaved}(a).
 The assertion that the class of all countably presentable
$\R$\+contramodules is closed under extensions is provable in
the same way as the similar assertion for modules over a ring.
 Hence the class $\sG$ of all countably presentable flat
$\R$\+contramodules is closed under extensions in $\R\Contra$.
 The class $\sG$ is also closed under kernels of epimorphisms in
$\R\Contra$ by Lemma~\ref{contramodule-countable-flat-coherence}.
 Finally, the class $\sG$ is self-generating in $\R\Contra$, since
every object of $\sG$ is a quotient object in $\R\Contra$ of an object
from $\sG$ (namely, of the free $\R$\+contramodule with a countable
set of generators) that is projective in $\R\Contra$.

 Part~(b): the assertion follows directly from
Corollary~\ref{count-pres-flat-contra-projdim1-cor}.
 One can also invoke Lemma~\ref{self-resolving-ext-isomorphism} here
to the effect that it follows from part~(a) that the Ext groups computed
in $\sG$ agree with the ones computed in $\R\Contra$.
\end{proof}

 The following theorem is a contramodule/ind-affine ind-scheme version
of~\cite[Theorem~2.4\,(1)\,$\Leftrightarrow$\,(3)]{EG} or
\cite[Theorem~8.6\,(ii)\,$\Leftrightarrow$\,(iii)]{Neem},
and an analogue of
Theorem~\ref{flat-comodule-pure-acyclic-complexes-characterized} above.
 It should be also compared to~\cite[Corollary~0.5 or
Proposition~8.13]{PS5}.
 In the latter context,
Theorem~\ref{flat-contramodule-pure-acyclic-complexes-charact-ed} can
be interpreted as saying that (over a topological ring $\R$ with
a countable topology base of two-sided ideals) \emph{for the exact
category of flat\/ $\R$\+contramodules, the derived category coincides
with the} (\emph{suitably defined}) \emph{coderived category}.

\begin{thm} \label{flat-contramodule-pure-acyclic-complexes-charact-ed}
 Let\/ $\R$ be a complete, separated topological ring with
a \emph{countable} base of neighborhoods of zero consisting of open
\emph{two-sided} ideals.
 Then the following classes of complexes of left\/ $\R$\+contramodules
coincide:
\begin{enumerate}
\item all pure acyclic complexes of flat\/ $\R$\+contramodules;
\item the closure of the class of all contractible complexes of
countably presentable flat\/ $\R$\+contramodules under extensions and
directed colimits;
\item all\/ $\aleph_1$\+directed colimits of totalizations of short
exact sequences of complexes of countably presentable flat\/
$\R$\+contramodules.
\end{enumerate}
\end{thm}

\begin{proof}
 The totalization of a short exact sequence of complexes is an extension
of the two cones of identity endomorphisms of the rightmost and
leftmost complexes in the short exact sequence.
 Furthermore, the class of all pure acyclic complexes of flat
$\R$\+contramodules is closed under extensions and directed colimits
(by Lemmas~\ref{flat-contramodules-well-behaved}(a)
and~\ref{directed-colimits-of-flat-contramodules-exact}), and
the contractible complexes of flat $\R$\+contramodules are pure acyclic.
 This proves the inclusions (3)\,$\subset$\,(2)\,$\subset$\,(1).

 Conversely, to prove that (1)\,$\subset$\,(3), recall that all pure
acyclic complexes of flat $\R$\+contramodules are $\aleph_1$\+directed
colimits of pure acyclic complexes of countably presentable flat
$\R$\+contramodules by 
Corollary~\ref{contramods-pure-acycl-cplxs-as-aleph1-dir-colims}.
 The exact category $\sG$ of countably presentable flat
$\R$\+contramodules has homological dimension~$\le1$
by Lemma~\ref{contramodules-count-pres-flat-exact-category}(b).
 Hence the pure acyclic complexes of countably presentable
flat $\R$\+contramodules are absolutely acyclic by
Lemma~\ref{acyclic-are-absolutely-acyclic}, and it remains to refer
to Proposition~\ref{absolutely-acyclic-are-summands-of-totalizations}
(for $d=1$).
\end{proof}

\begin{qsts}
 The main results of this paper concerning contramodules, such as
Theorems~\ref{flat-contramodules-as-directed-colimits-theorem},
\ref{contramodule-cotorsion-periodicity-theorem}, and
and~\ref{flat-contramodule-pure-acyclic-complexes-charact-ed},
as well as Propositions~\ref{complexes-of-flat-contramodules-prop}
and~\ref{short-exact-sequences-of-flat-contramodules}, and
Corollary~\ref{contramods-pure-acycl-cplxs-as-aleph1-dir-colims}
assume existence of a countable topology base of two-sided ideals
in the topological ring~$\R$.
 Are any of these theorems valid for complete, separated topological
rings with a countable base of neighborhoods of zero consisting of
open \emph{right} ideals?
 It would be interesting to know; but in fact, even
Questions~\ref{count-pres-flat-contra-projdim-base-of-right-qst}
and~\ref{contramodule-flat-coherence-base-of-right-qst} seem to be open.
\end{qsts}

\end{document}